\documentclass[11pt,reqno]{amsart}
\usepackage{graphics, color, amsmath, amsfonts, amssymb, amsthm, latexsym}
\usepackage{times}
\usepackage[latin1]{inputenc}
\usepackage[T1]{fontenc}

\usepackage[dvips]{graphicx}
\usepackage[margin=1in]{geometry}

\usepackage{hyperref}

\newcommand{\bea}{\begin{eqnarray}}
\newcommand{\eea}{\end{eqnarray}}
\def\beaa{\begin{eqnarray*}}
\def\eeaa{\end{eqnarray*}}
\def\ba{\begin{array}}
\def\ea{\end{array}}
\def\be#1{\begin{equation} \label{#1}}
\def \eeq{\end{equation}}

\newcommand{\nn}{\nonumber}

\def\be{{\beta}}

\def\e{\varepsilon}

\def\s{\sigma}

\def\les{\lesssim}

\def\Z{{\mathbb{Z}}}
\def\R{{\mathbb{R}}}

\def\F{{\mathcal{F}}}

\def\N{{\mathcal{N}}}

\def\p{{\prime}}

\def\what{\widehat}
\theoremstyle{plain}

\newtheorem{theorem}{Theorem}[section]
\newtheorem{lem}[theorem]{Lemma}
\newtheorem{proposition}[theorem]{Proposition}

\numberwithin{equation}{section}

\begin{document}

\title[Modified scattering for the Boson Star Equation]{Modified scattering for the Boson Star Equation}

\author[Fabio Pusateri]{Fabio Pusateri}
\address{Princeton University}
\email{fabiop@math.princeton.edu}

\thanks{The author was supported in part by a Simons Postdoctoral Fellowship.}

\begin{abstract} 
We consider the question of scattering for the boson star equation in three space dimensions.
This is a semi-relativistic Klein-Gordon equation with a cubic nonlinearity of Hartree type.
We combine weighted estimates, obtained by exploiting a special null structure present in the equation,
and a refined asymptotic analysis performed in Fourier space,
to obtain global solutions evolving from small and localized Cauchy data.
We describe the behavior at infinity of such solutions by identifying a suitable nonlinear asymptotic correction to scattering.
As a byproduct of the weighted energy estimates alone, 
we also obtain global existence and (linear) scattering for solutions of semi-relativistic Hartree equations
with potentials decaying faster than Coulomb.
\end{abstract}

\maketitle
\setcounter{tocdepth}{1}
\tableofcontents

\section{Introduction}

\vskip10pt
\subsection{The Equation}

We consider the semi-relativistic Klein-Gordon equation with a cubic Hartree-type nonlinearity
\begin{align}
\label{eq}
i \partial_t u  -  \sqrt{m^2-\Delta} u =  \lambda \left( {|x|}^{-1} \ast {|u|}^2 \right) u \, ,
\end{align}
with $u: (t,x) \in \R \times \R^3 \rightarrow \mathbb{C}$, and $m,\lambda \in \R$.
The operator $\sqrt{m^2-\Delta}$ is defined as usual by its symbol $\sqrt{m^2+|\xi|^2}$ in Fourier space,
and $\ast$ denotes the convolution on $\R^3$.
In theoretical astrophysics, \eqref{eq} is used to describe the dynamics of boson stars (Chandrasekhar theory),
and it is often referred to as the {\it boson star} equation.
%In \cite{ES}, Elgart and Schlein considered a quantum mechanical system of $N$ bosons with relativistic dispersion,
%interacting through a mean field Coulomb potential.
In \cite{ES}, Elgart and Schlein rigorously derived \eqref{eq}
via the mean field theory for quantum many-body systems of boson particles %(boson star) 
with Coulomb type (gravitational) interaction.
%\eqref{eq} is therefore meant to govern the time evolution of the density of a relativistic boson particle with mass $m$.
In the past few years the semi-relativistic equation \eqref{eq} %has raised considerable mathematical interest,
and has been analyzed by several authors with regards to various aspects of the PDE theory.
We will discuss some of the most relevant works on \eqref{eq},
and on some of its generalizations, in section \ref{background} below.
%let us informally describe the main result of this paper.
In this paper we are interested in the asymptotic behavior as $t \rightarrow \infty$ of small solutions 
of the Cauchy problem associated to \eqref{eq}, and, in particular, in the question of scattering.
Our main result is the following: 
{\em For any given $u_0(x) = u(t=0,x)$ which is small enough in a suitable weighted Sobolev space,
there exists a unique global solution of \eqref{eq} which decays pointwise over time like a solution of the linear equation,
but, as time goes to infinity, scatters in a nonlinear fashion}.
This phenomenon of nonlinear (modified) scattering happens similarly for the standard Hartree equation \cite{HN,KP}
$$i \partial_t u  -  \Delta u =  \left( {|x|}^{-1} \ast {|u|}^2 \right) u \quad , \quad x \in \R^n \, , \, n\geq 2 \, ,$$
and, in its essence, it is the same type of asymptotic behavior that can be found 
in several others dispersive equations which are scattering-critical (or $L^\infty$-critical) .
%We shall discuss this aspect in more details below. 
An additional result contained in the present paper concerns some generalizations of \eqref{eq} 
with potentials decaying faster than the Coulomb potential ${|x|}^{-1}$.
We will prove (regular) scattering for those models, closing some gaps in the existing literature.

\subsection{Background and known results}\label{background}
As pointed out above, the semi-relativistic equation \eqref{eq} can be
rigorously derived as the mean field limit of an $N$-body system of interacting boson particles.
In the time independent case, the question of convergence and existence of solutions for the limiting equation 
had been studied earlier by Lieb and Yau \cite{LY}.
More recent investigations on the relation between the $N$-particle system 
and the limiting nonlinear equation \eqref{eq}, can be found in \cite{MS}.

The conserved energy associated to \eqref{eq} is
\begin{align}
 \label{energy}
E(u) := \frac{1}{2} \int_{\R^3} \bar{u} \sqrt{m^2-\Delta} u \, dx 
  +  \frac{\lambda}{4}  \int_{\R^3} \left( {|x|}^{-1} \ast {|u|}^2 \right) {|u|}^2 \, dx \, ,
\end{align}
and therefore the energy space  is $H^{1/2}$.
Solutions of \eqref{eq} also enjoy conservation of mass, ${\| u(t) \|}_{L^2} = {\| u(0) \|}_{L^2}$,
and the nonlinearity $( {|x|}^{-1} \ast {|u|}^2 ) u$ is critical with respect to $L^2$ in three dimensions.

Local existence of solutions for the Cauchy problem with data in $H^s(\R^3)$, $s \geq 1/2$,
was proved by Lenzmann in \cite{L1}, also for more general models than \eqref{eq}, including a wide class of external potentials.
In the cited paper, using conservation of energy, global existence is obtained for any data in the defocusing case $\lambda \geq 0$.
In the focusing case $\lambda < 0$, one needs instead to restrict the size of the $L^2$-norm of the initial data 
to be smaller than that of the ground state \cite{FraLen1}. 
It was shown by by Fr\"{o}lich and Lenzmann \cite{FL2}
that, in the focusing case, any radially symmetric smooth compactly supported initial data with negative energy leads to finite time blow-up.
%
%Additional analysis of the blow-up can be found in \cite{LenLev}.
%
Sharp low regularity wellposedness below the energy space was recently proven by Herr and Lenzmann \cite{HL}, 
both in the radial ($s>0$) and non-radial case ($s \geq 1/4$). %Both results are sharp in the scale of Sobolev spaces.
%
%Existence and stability of traveling solitary waves (case $\lambda < 0$) have been investigated in \cite{FJL1,FJL2}.
%
%%%
%For additional works on wellposedness, blowup, and existence and stability of solitary waves for \eqref{eq} and related models,
%we refer to \cite{FL3,LenLev,FJL1,FJL2} and references therein.
%%%

Without loss of generality we can normalize $m=1$, and rescale $\lambda$ to be $1$ or $-1$ depending on its sign.
In this paper we will only consider small solutions,
and therefore the sign of $\lambda$ will not be relevant, and $\lambda$ will be taken to be $-1$ for convenience.
To better put \eqref{eq} into context in relation to the global well-posedness and scattering theory for the Cauchy problem,
let us consider the following generalized model
\begin{align}
\label{eqgamma}
i\partial_t u - \sqrt{1-\Delta} u = - \left( {|x|}^{-\gamma} \ast {|u|}^2 \right) u  \quad , \quad x \in \R^n \quad , \quad 0 < \gamma < n \, .
\end{align}
In \cite{COSIAM06,COJKMS07}, Cho and Ozawa showed global existence of large solutions for $0 < \gamma < 2n/(n+1)$ for $n \geq 2$,
and small data global existence and scattering for $\gamma > 2$ in dimension $n \geq 3$.
They also proved the non-existence of asymptotically free solutions (i.e. solutions converging to a solution of the linear equation)
for the case $0 < \gamma \leq  1$ when $n\geq 3$, and for $0 < \gamma < n/2$ when $n=1$ or $2$.
Our main result shows that indeed solutions of \eqref{eqgamma} with $\gamma = 1$ in $3d$ scatter to a nonlinear profile.
An additional result that we prove, namely Theorem \ref{theoHg}, closes the gap in the 
small data scattering\footnote{At least for a class of initial data in a suitable weighted Sobolev space.}\footnote{Notice that there 
seem to be no global solutions in the literature, in the intermediate range $3/2 \leq \gamma \leq 2$.}
%Even with radial symmetry only the case $\gamma < 5/3$ has been treated in the literature} 
for $1 < \gamma \leq 2$.

The large data global existence results above were subsequently improved by the same authors \cite{CODCDS08}, 
in the radially symmetric case, to include $1 < \gamma < (2n-1)/n$.
In \cite{COSSDCDS09} the authors obtained scattering for radially symmetric small solutions when $3/2 < \gamma < 2$ and $n \geq 3$.
Cho and Nakanishi \cite{CNRIMS10} obtained several results in higher dimensions:
in dimension $n \geq 4$ they proved global existence with radial symmetry for $1 < \gamma < 2$,
and small data scattering (also without symmetry) for $\gamma = 2$.
We refer to \cite{CNRIMS10} for a survey of some of the techniques employed in the above mentioned papers.

\subsection{Main Result}
We have seen that for certain values of $0<\gamma<2$, large global solutions to \eqref{eqgamma} can be constructed combining conservation laws, 
and low regularity wellposedness or Strichartz estimates (and Hardy's inequalities in order to estimate the nonlinearity).
When the question of scattering is considered, even treating small data outside the energy space is quite challenging.
As mentioned above, in three dimensions, scattering is known if $\gamma$ is large enough 
($\gamma > 3/2$ in the radial case, $\gamma >2$ in the general case).
Clearly, larger values of $\gamma$ are easier to treat, since the time decay of the $L^2$ norm of the nonlinearity in \eqref{eqgamma}, 
computed on a solution of the linear equation, is $t^{-\gamma}$.
Theorem \ref{maintheo} below shows scattering (in a modified sense) for \eqref{eq}, that is \eqref{eqgamma} with $\gamma = 1$. 
We refer to this as the ``scattering-critical'' case, because the decay of the nonlinearity is (barely) non-integrable in time.
Moreover, our proof can be adapted to obtain scattering in the $L^\infty$-subcritical cases $1 < \gamma \leq 2$, which were left open so far. 
See Theorem \ref{theoHg} for details.

This is our main result:
\begin{theorem}\label{maintheo}
Let\footnote{For convenience, and to simplify the proof a bit, we let $N$ be comfortably large;
however, it is certainly possible to reduce the value of $N$ to a number between $10$ and $100$.} 
$N=1000$, %and $p_0 = 1/1000$. 
and let $u_0 : \R^3 \rightarrow \mathbb{C}$ be given such that
\begin{align}
\label{initdata}
{\| u_0 \|}_{H^N} + {\| {\langle x \rangle}^2 u_0 \|}_{H^2} + {\| {(1+|\xi|)}^{10} \what{u_0} \|}_{L^\infty} \leq \e_0 \, .
\end{align}
Then there exists $\bar{\e}_0$ such that for all $\e_0 \leq \bar{\e}_0$,
the Cauchy problem
\begin{equation}
\label{eq1}
\left\{
\begin{array}{l}
i \partial_t u - \sqrt{1-\Delta} u =  -\big( {|x|}^{-1} \ast {|u|}^2 \big) u 
\\
u(t=0,x) = u_0(x)
\end{array}
\right.
\end{equation}
has a unique global solution $u(t,x)$, such that
\begin{align}
\sup_{t \in \R} {(1+|t|)}^{-3/2} {\| u(t) \|}_{L^\infty} \les \e_0 \, .
\end{align}
%Moreover, if we denote $f(t,x) := ( e^{it\sqrt{1-\Delta}} u(t) ) (x)$, then
%\begin{align}
%\begin{split}
%\sup_{t \in \R} \Big[ {(1+t)}^{-p_0} {\| u(t) \|}_{H^N} + 
%  {(1+t)}^{-p_0} {\left\| x f(t) \right\|}_{H^1} + {(1+t)}^{-2 p_0} {\left\| x^2 f(t) \right\|}_{H^2}
%+ {\big\| {(1+|\xi|)}^{10} \what{f}(t) \big\|}_{L^\infty}  \Big] \leq C \e_0 \, .
%\end{split}
%\end{align}
Moreover, the behavior of $u$ as $t \rightarrow \infty$ can be described as follows.
Let
\begin{align}
\label{maintheocorr}
B(t,\xi) :=  \frac{1}{{(2\pi)}^3} \int_0^t \int_{\R^3} {\left| \frac{\xi}{\langle \xi \rangle} - \frac{\s}{\langle \s\rangle} \right|}^{-1} 
  {|\what{u}(s,\s)|}^2 \, d\s \, \varphi (\xi s^{-1/300}) \frac{ds}{s+1} \, ,
\end{align}
where $\varphi$ is a smooth compactly supported function.
Then, there exists an asymptotic state $f_+$, such that for all $t>0$
\begin{align}
\label{maintheoscatt}
{\big\| {(1+|\xi|)}^{10} \big[ e^{iB(t,\xi)} e^{it \sqrt{1+{|\xi|}^2}} \what{u}(t,\xi) 
  - f_+(\xi) \big] \big\|}_{L^\infty_\xi} \les \e_0 {(1+t)}^{-p_1} \, ,
\end{align}
for some $0 < p_1 < 1/1000$. A similar statement holds for $t<0$.
\end{theorem}

Solutions of \eqref{eq1} will be constructed through a priori estimates in the space given by the norm \eqref{norm}.
We refer to section \ref{secideas} for some explanation of the main ideas involved the proof of Theorem \ref{maintheo},
and to section \ref{secstra} for a detailed description of our strategy.

It would be possible to express the asymptotic behavior of a solution of \eqref{eq1}
in physical coordinates rather than in Fourier space.
However, the asymptotic formula \eqref{maintheocorr}-\eqref{maintheoscatt}
clearly emerges from our proof, which is performed in Fourier space, 
and can be seen from some heuristic considerations, see section \ref{secideas}.
Therefore we leave \eqref{maintheocorr}-\eqref{maintheoscatt} as a satisfactory description of modified scattering.

Before moving on to describe the difficulties and the tools involved in the proof of Theorem \ref{maintheo},
let us mention some known results concerning modified scattering.
Famous examples of dispersive PDEs whose solutions exhibit a behavior which is qualitatively different 
from the behavior of a linear solution are the nonlinear Sch\"rodinger \cite{MZ,DZNLS,HN,KP}, 
the Benjamin-Ono \cite{AFBO,HNBO}, and the mKdV \cite{DZmKdV,HNmKdV} equations.
Besides these one-dimensional completely integrable examples, for which large data results are also available, 
the phenomenon of modified scattering for small solutions has been observed in several other equations.
Example are given by Hartree equations \cite{HN,KP}, Klein-Gordon equations \cite{DelortKG1d}, 
and, more recently, gravity water waves \cite{2dWW}
(see also \cite{FNLS} for a simpler fractional Sch\"{o}dinger model, and \cite{AD} for a similar result on the water waves system).

%In essence, it is the type of asymptotic behavior of solution of well-known equations such as
%the nonlinear Sch\"rodinger equation \cite{MZ,DZNLS,HN,KP} ,
%as well as other non-integrable dispersive (quasilinear) PDEs, such as the Klein-Gordon equation \cite{DelortKG1d} 
%and the 2d gravity water waves system \cite{2dWW}. 
%See \cite{FNLS} for a simplified water waves model, and \cite{AD} for a global existence and scattering result similar to \cite{2dWW}.

%Recently, modified scattering was shown for the first time in the context of quasilinear systems
%in the work by the author and Ionescu on the 2d gravity water waves equation.

\vskip15pt
\subsection*{Notations}
We define the Fourier transform by
\begin{align*}
\mathcal{F} g (\xi) = \what{g}(\xi) := \int_{\R^3} e^{-ix \cdot \xi} g(x) \,dx 
  \quad \implies \quad g(x) = \frac{1}{{(2\pi)}^{3}} \int_{\R^3} e^{ix \cdot \xi} \what{g}(\xi) \,d\xi \, .
\end{align*}
We fix $\varphi:\mathbb{R}\to[0,1]$ an even smooth function supported in $[-8/5,8/5]$ and 
equal to $1$ in $[-5/4,5/4]$, and let
\begin{equation}
 \label{phi_k}
\varphi_k(x):=\varphi(|x|/2^k)-\varphi(|x|/2^{k-1}) \, , \, \qquad k \in \Z \, , \, x \in \R^3 \, .
\end{equation}
For any interval $I\subseteq\mathbb{R}$ we define
\begin{equation}\label{phi_k1}
\varphi_I := \sum_{k\in I\cap\mathbb{Z}} \varphi_k  \, . %\qquad \varphi^{(m)}_I:=\sum_{k\in I\cap\mathbb{Z}\cap[m,\infty)}\varphi^{(m)}_k \, .
\end{equation}
More generally, for any $m,k\in\mathbb{Z}$, $m\leq k$, and $x \in \R^3$ we define
\begin{equation}\label{phi^m_k}
 \varphi^{(m)}_k(x):=
\begin{cases}
\varphi(|x|/2^k) - \varphi(|x|/2^{k-1}),\qquad & \text{ if } k \geq m+1 \, ,
\\
\varphi(|x|/2^k),\qquad &\text{ if } k=m \, .
\end{cases}
\end{equation}
We let $P_k$, $k\in\mathbb{Z}$, denote the operator on $\R^3$ defined by the Fourier multiplier $\xi\to \varphi_k(\xi)$.
We will sometimes denote $f_k = P_k f$.
For an integer $n\in\Z$ we denote $n_+ = \max(0,n)$.

\vskip15pt
\section{Main ideas}\label{secideas}
Let $p_0 = 1/1000$, $N = 1000$ and $\Lambda(\nabla) := \sqrt{1-\Delta}$. 
Define
\begin{align}
 \label{prof}
f(t,x) := (e^{it \Lambda} u )(t,x) \, ,
\end{align}
where $u(t)$ is a solution of \eqref{eq1}.
We will solve \eqref{eq1} in the space given by the norm
\begin{align}
\label{norm}
\begin{split}
%&  {\|u\|}_X \\ & := 
\sup_t \left[ {(1+t)}^{-p_0} {\| u(t) \|}_{H^N} + 
  {(1+t)}^{-p_0} {\left\| x f(t) \right\|}_{H^2} + {(1+t)}^{-2 p_0} {\left\| x f(t) \right\|}_{H^2}
  + {\big\| {(1+|\xi|)}^{10} \widehat{f}(t) \big\|}_{L^\infty}  \right] \, .
\end{split}
\end{align}
%
%\vskip10pt
%\subsection*{Resonances analysis}
If $f$ is defined as in \eqref{prof}, we can write Duhamel's formula for \eqref{eq} in Fourier space as follows:
\begin{align}
 \label{inteq}
\begin{split}
\widehat{f}(t,\xi) & = \widehat{u_0}(\xi) +  \int_0^t I(s,\xi) \, ds \, ,
\\
I(s,\xi) & := i c_1 \iint_{\R^3 \times \R^3} e^{is [-\Lambda(\xi)+\Lambda(\xi-\eta)+\Lambda(\eta+\s)-\Lambda(\s)]} {|\eta|}^{-2}
     \what{f}(s,\xi-\eta) \widehat{f}(s,\eta+\s) \overline{\what{f}}(s,\sigma) \, d\eta d\s  \, ,
\\
c_1 & :=  2 {(2\pi)}^{-5} \, . 
\end{split}
\end{align}
Here we used $\F (|x|^{-1}) (\xi) = 4 \pi {|\xi|}^{-2}$.

\vskip10pt
\subsection*{Norms and decay}
When dealing with nonlinear equations which are scattering-critical in the sense explained in the introduction,
one often has to resort to spaces which incorporate some strong decay information.
Even more so if modified scattering is expected.
In this case, one needs to extract very precise asymptotic information, 
and be able to prove some $L^\infty_t L^p_x$ bound on solutions.

Thanks to the decay estimate \eqref{disperse0}, one sees that solutions whose norm \eqref{norm} is bounded,
decay pointwise like a solution of the linear equation, i.e. at the rate of $t^{-3/2}$.
It is important to underline the key role played by the $\F^{-1} L^\infty$-norm.
Since the equation \eqref{eq1} is $L^\infty$-critical, 
one might {\em not} be able to prove a bound on weighted $L^2$-norms of $f$ which is uniform in time.
Therefore, sharp time decay cannot be obtained as a consequence of standard (weighted) $L^p-L^q$ linear estimates.
The idea, already exploited in several works on other critical models,
is to include a norm which guarantees decay, but is weaker than $L^1$, and as such can be controlled uniformly in time.
Our choice is the last norm appearing in \eqref{norm}.

\subsection*{Weighted Estimates}
Weighted norms play an important role in the whole construction.
Firstly, they control remainders in the linear estimate \eqref{disperse0}.
Secondly, and most importantly,
they are used to control remainders in the asymptotic expansions which allow us to bound the $\F^{-1} L^\infty$-norm
(see the next paragraph below for more on this).
In the literature, a standard way of establishing weighted estimates is given by the use of vectorfields \cite{K0,K1}.
In the case of the boson star equation, the use of such a tool is limited by the lack of scaling and Lorentz invariance.
The quantities we shall control are $xf$ and $x^2f$ in $L^2$. These correspond to $\Gamma u$ and $\Gamma^2 u$ in $L^2$, 
for $\Gamma = x - it\Lambda^\p$.
Despite the fact that $\Gamma$ does not commute properly with the equation,
we will be able to bound these weighted norms as follows.
We apply $\nabla_\xi$ and $\nabla_\xi^2$ to $\what{f}$ as given in \eqref{inteq}. 
The worst term obtained by applying $\nabla_\xi$ to $I(s,\xi)$ is of the form
\begin{align*}
\iint_{\R^3 \times \R^3} e^{is [-\Lambda(\xi)+\Lambda(\xi-\eta)+\Lambda(\eta+\s)-\Lambda(\s)]} \, s m(\xi,\eta){|\eta|}^{-2}
     \what{f}(s,\xi-\eta) \widehat{f}(s,\eta+\s) \overline{\what{f}}(s,\sigma) \, d\eta d\s  \, ,
\end{align*}
where $m(\xi,\eta) = \nabla_\xi(-\Lambda(\xi) + \Lambda(\xi-\eta))$.
Now notice that $m$ is a smooth function with $m(\xi,0) = 0$.
This is essentially a null condition satisfied by the equation\footnote{This type of generalized null condition 
was used in \cite{GMS2} and \cite{nullcondition}, as an important aspect of the space-time resonance analysis.}.
Thanks to this, we can think that the multiplier $s m(\xi,\eta){|\eta|}^{-2}$ behaves, 
as far as estimates are concerned, like the original Coulomb potential ${|\eta|}^{-2}$,
so that the loss of the factor $s$ can be recovered\footnote{Another 
alternative possibility is to proceed similarly to \cite{nullcondition} and \cite{zakharov}, 
by exploiting an algebraic identity for the phase $\phi (\xi,\eta,\s) = -\Lambda(\xi)+\Lambda(\xi-\eta)+\Lambda(\eta+\s)-\Lambda(\s)$, 
of the form $\nabla_\xi \phi = L(\nabla_\eta \phi, \nabla_\s \phi, \phi)$, 
where $L$ denotes some linear combination with coefficients given by smooth functions of $(\xi,\eta,\s)$.
One could use such an identity to integrate by parts in time and frequency and recover the loss of $s$.}.
We can then control $xf$ and $x^2f$ in $L^2$, allowing a small growth in $t$.

\vskip10pt
\subsection*{Asymptotic analysis}
Let us change variables in \eqref{inteq} and write
\begin{align}
 \label{inteq1}
%\begin{split}
%\widehat{f}(t,\xi) & = \widehat{u_0}(\xi) +  \int_0^t I(s,\xi) \, ds \, ,
%\\
I(s,\xi) & := i c_1 \iint e^{is [-\Lambda(\xi)+\Lambda(\xi+\eta)+\Lambda(\xi+\s)+\Lambda(\xi+\eta+\s)]} {|\eta|}^{-2}
     \what{f}(s,\xi+\eta) \widehat{f}(s,\xi+\s) \overline{\what{f}}(s,\xi+\eta+\s) \, d\eta d\s  \, . 
%\end{split}
\end{align}
Our goal is to identify the leading order term of the above expression in terms of powers of $s$,
neglecting all contributions that decay faster than $s^{-1}$.
Let us assume that $|\xi| \sim 1$ and we are integrating on a region $|\eta| \les s^{l}$, with $l < 0$ small enough, 
but not so small that the integral of $|\eta|^{-2}$ over this region is $O(s^{-1-})$.
We can then Taylor expand the oscillating phase, and approximate $I(s,\xi)$ by
\begin{align*}
 %\label{inteqas1}
i c_1 \iint e^{is \eta \cdot z} {|\eta|}^{-2}
     \what{f}(s,\xi+\eta) \widehat{f}(s,\xi+\s) \overline{\what{f}}(s,\xi+\eta+\s) \, d\eta d\s  \, ,
\end{align*}
where $z = z(\xi,\s) := \s/\langle \s \rangle - \xi/\langle \xi \rangle$.
Using the bounds on $\partial \what{f}$, i.e. on weighted norms, we can further approximate the expression above by
\begin{align*}
% \label{inteqas2}
& i c_1 \iint e^{is \eta \cdot z} {|\eta|}^{-2}
     \what{f}(s,\xi) \widehat{f}(s,\xi+\s) \overline{\what{f}}(s,\xi+\s) \, d\eta d\s  
\\
= \, \, & i c_1 \what{f}(s,\xi) \int \F({|\eta|}^{-2})(s z) \, {|\what{f}(s,\xi+\s)|}^2 \, d\s 
= \frac{i}{s} \what{f}(s,\xi) C(s,\xi) \, ,
\end{align*}
for some function $C(s,\xi)$ which is real-valued, and uniformly bounded under suitable assumptions on $\what{f}$.
Thanks to the above we have obtained 
\begin{align*}
\partial_t \what{f}(t,\xi) = i t^{-1} \what{f}(t,\xi) C(t,\xi) + O(t^{-1-}) \, ,
\end{align*}
from which we can deduce a uniform bound on $\sup_{t,\xi} |\what{f}(t,\xi)|$.
The estimates leading to this latter bound will also show the modified scattering property \eqref{maintheoscatt}.

In order to make the above intuition rigorous, we need to identify a suitable scale in $\eta$, say $s^{l_0}$,
such that the above asymptotics are true for $|\eta| \les s^{l_0}$,
and, at the same time, the integral \eqref{inteq1} on the region $|\eta| \gtrsim s^{l_0}$ is $O(s^{-1-})$.
Lemma \ref{lemcorr} contains the derivation of the asymptotic correction term in the critical region.
The remaining contributions are estimated in section \ref{secrem}, using integration by parts in $\eta$.

%This latter constitutes the most technical part of the paper.
%We will decompose dyadically all the profile and the region of integration.
%Then, using the fact that $|\eta|$ is not too small in a suitable sense,
%we will integrate by parts in $\eta$ and show that all remainder contributions are integrable in time.

\vskip15pt
\section{Strategy of the proof}\label{secstra}
Local-in-time solutions to \eqref{eq1} can be constructed by a standard fixed point argument.
Given a local solution $u$ on a time interval $[0,T]$, we assume that the following norm is a priori small:
\begin{align}
\label{apriori}
\begin{split}
 {\|u\|}_{X_T} := \sup_{t \in [0,T]} \Big[ {(1+t)}^{-p_0} {\| u(t) \|}_{H^N} + 
  {(1+t)}^{-p_0} {\left\| x f(t) \right\|}_{H^2} + {(1+t)}^{-2 p_0} {\left\| x^2 f(t) \right\|}_{H^2}
\\
+ \,\, {\big\| {(1+|\xi|)}^{10} \widehat{f}(t,\xi) \big\|}_{L^\infty_\xi}  \Big] \leq \e_1 \, .
\end{split}
\end{align}
To obtain the existence of a global solution which is bounded in the space $X_T$ it will suffice to show
\begin{align}
\label{conc}
 {\| u \|}_{X_T} \leq \e_0 + C \e_1^3 \, ,
\end{align}
where $\e_0$ is the size of the initial datum, see \eqref{initdata}.

In order to deduce sharp pointwise decay from the above a priori bounds we will use the following:
\begin{proposition}[Refined Linear Decay Estimate]\label{prodecay}
  For any $t\in\mathbb{R}$ one has
\begin{align}
\label{disperse0}
{\big\|e^{i t \sqrt{1-\Delta}} f \big\|}_{L^\infty} 
  \les \frac{1}{(1+|t|)^{3/2}} {\big\| {(1+|\xi|)}^6 \what{f}(\xi) \big\|}_{L^\infty_\xi}
  + \frac{1}{(1+|t|)^{31/20}} \Big[ {\big\| {\langle x \rangle}^2 f \big\|}_{L^2} + {\|f\|}_{H^{50}} \Big] \, .
\end{align}
\end{proposition}

\noindent
Proposition \ref{prodecay} is proven in section \ref{seclinest}.
As a consequence of \eqref{disperse0} and the a priori assumptions \eqref{apriori} we have for $t\in[0,T]$
\begin{align}
 \label{aprioridecay}
{\| u(t) \|}_{W^{2,\infty}} \les \e_1 {(1+t)}^{-3/2} \, .
\end{align}

The proof of \eqref{conc} will be done in two main steps given by the following Propositions.

\begin{proposition}[Weighted Energy Estimates]\label{proW}
Assume that $f\in C([0,T]:H^N)$ satisfies the a priori assumptions \eqref{apriori}, and let $p_0 = 1/1000$. Then, 
\begin{align*}
& \sup_{t \in[0,T] } {(1+t)}^{-p_0} {\left\| f(t) \right\|}_{H^N} \leq \e_0 + C \e_1^3 \, ,
\end{align*}
and
\begin{align*}
& \sup_{t \in[0,T] } {(1+t)}^{-p_0} {\left\| \langle x \rangle f(t) \right\|}_{H^2} \leq \e_0 + C \e_1^3 \, ,
\\
& \sup_{t \in[0,T] } {(1+t)}^{- 2p_0} {\big\| {\langle x \rangle}^2 f(t) \big\|}_{H^2} \leq \e_0 + C \e_1^3 \, .
\end{align*}
\end{proposition}

\noindent
The proof of the above Proposition is contained in section \ref{secweighted}.

\begin{proposition}[Estimate of the $L^\infty_\xi$-norm]\label{proZ}
Assume that $f\in C([0,T]:H^{N})$ satisfies the a priori bounds \eqref{apriori}. Then %, for some $p_1>0$,
%\begin{align}
%\label{keybound00}
%\sup_{t_1 \leq t_2 \in[0,T]} {(1+t_1)}^{p_1} 
%  {\left\| {(1+|\xi|)}^{3} \left( g(t_2,\xi) - g(t_1,\xi) \right) \right\|}_{L^\infty_\xi} \les \e_1^3 \, .
%\end{align}
%In particular
\begin{align*}
\sup_{t \in[0,T] } {\big\| {(1+|\xi|)}^{10} \what{f}(t,\xi) \big\|}_{L^\infty_\xi} \leq \e_0 + C \e_1^3 \, .
\end{align*}
\end{proposition}

As a corollary of the proof of Proposition \ref{proZ} we will obtain the main result of our paper
concerning scattering of small solution of \eqref{eq1}:

\begin{proposition}[Modified Scattering]\label{proscatt}
Assume that $u \in C([0,T]:H^{N})$ is a solution of \eqref{eq1} with $u_0$ small enough as in \eqref{initdata}.
Define
\begin{align}
\label{correction0}
B(t,\xi) :=  {(2\pi)}^{-3} \int_0^t \int_{\R^3} {\left| \frac{\xi}{\langle \xi \rangle} - \frac{\s}{\langle \s\rangle} \right|}^{-1} 
  {|\what{u}(s,\s)|}^2 \, d\s \, \varphi (\xi s^{-1/300}) \frac{ds}{s+1} \, ,
\end{align}
where $\varphi$ is a smooth compactly supported function as the one described before \eqref{phi_k}.
Then, there exists $f_+ \in L^\infty_\xi$, and $p_1 >0$, such that 
\begin{align}
\label{estproscatt}
{\big\| {(1+|\xi|)}^{10} \big( e^{iB(t,\xi)} \what{f}(t,\xi) %e^{it\Lambda(\xi)} \what{u}(t,\xi) 
  - f_+(\xi) \big) \big\|}_{L^\infty_\xi} \les \e_1^3 {(1+t)}^{-p_1} \, .
\end{align}
\end{proposition}

\noindent
We refer to section \ref{secLinfty} for the proof of Propositions \ref{proZ} and \ref{proscatt}.
In particular, Proposition \ref{prohatf} implies both Propositions \ref{proZ} and \ref{proscatt}, through the estimate \eqref{keybound0}.

\vskip15pt
\section{Proof of Proposition \ref{proW}}\label{secweighted}

Let us define 
%the linear operator by 
%\begin{align}\label{L}
%\L := i\partial_t - \sqrt{1-\Delta} \, ,
%\end{align}
%and
\begin{align}
\label{N}
\N(h_1, h_2, h_3) := \big( {|x|}^{-1} \ast h_1 \overline{h_2} \big) h_3 \, .
\end{align}
%\eqref{eq} then becomes 
%\begin{align}\label{eq2} 
%\L u = \N (u,u,u) \, .
%\end{align}
%
%
%\subsection{Estimates of Sobolev norms}
A priori energy estimates for the Sobolev norms of solutions to \eqref{eq1} are straightforward.
In particular, using the Hardy-Littlewood-Sobolev inequality, 
it is not hard to show that, given a solution $u: [0,T] \times \R^3$, for all $t \in [0,T]$ one has
\begin{align}
\label{ee}
 {\| u (t) \|}_{H^N} \leq {\| u_0 \|}_{H^N} + C \int_0^t {\| u(s) \|}_{L^2} {\| u(s) \|}_{L^6} {\| u (s) \|}_{H^N} \, ds
\end{align}
for all integers $N \geq 0$, and some constant $C>0$.
Interpolating the a priori decay assumption in \eqref{apriori} and the bounds on the $L^\infty_\xi$-norm (which controls the $L^2_x$-norm), 
one has  ${\| u(s) \|}_{L^6} \lesssim \e_1 {(1+s)}^{-1}$. 
Using the a priori assumptions \eqref{apriori} it follows from \eqref{ee} that 
\begin{align}
 {\| u (t) \|}_{H^N} \leq \e_0 + C \e_1^3 {(1+t)}^{p_0} \, .
\end{align}

\vskip5pt
%Define the weighted norms:
%\begin{align}
%{\| u \|}_{W_j} = {\| {\langle x \rangle}^k f \|}_{H^k}
%\end{align}
%where $f = e^{-it\sqrt{1-\Delta}} u$.
To obtain Proposition \ref{proW}, we then aim to prove that under the a priori assumptions \eqref{apriori} one has
\begin{align}
\label{estW1}
 & {\| \langle x \rangle f(t) \|}_{H^2} \leq \e_0 + C \e_1^3 {(1+t)}^{p_0} \, ,
\\
\label{estW2}
 & {\| {\langle x \rangle}^2 f(t) \|}_{H^2}  \leq \e_0 + C \e_1^3 {(1+t)}^{2p_0} \, .
\end{align}

\vskip15pt
\subsection{Proof of \eqref{estW1}}\label{secweighted1}
Since we already have control on the $L^2$-norm of $f$ we just need to estimate $xf$ in $H^2$.
Recall the integral equation satisfied by $f$:
\begin{align}
 \label{inteqW1}
\begin{split}
\widehat{f}(t,\xi) & = \widehat{u_0}(\xi) +  \int_0^t I(s,\xi) \, ds \, ,
\\
I(s,\xi) & := i c_1 \iint_{\R^3 \times \R^3} e^{is \phi(\xi,\eta,\sigma)} {|\eta|}^{-2}
     \widehat{f}(s,\xi-\eta) \widehat{f}(s,\eta+\sigma) \overline{\widehat{f}}(s,\sigma)  \, d\eta d\sigma  \, ,
\\
& \phi(\xi,\eta,\sigma) := - \Lambda(\xi)  + \Lambda(\xi-\eta) + \Lambda(\eta+\sigma) - \Lambda(\sigma) \, . 
\end{split}
\end{align}
%In what follows we will sometime omit the time variable $s$ as an argument of $\what{f}$.
For \eqref{estW1} it suffices to prove that under the a priori assumptions \eqref{apriori} we have
\begin{align}
\label{concW1}
& {\Big\| \langle \xi \rangle \partial_\xi I(s)  \Big\|}_{L^2} \les \e_1^3 {(1+s)}^{-1+p_0} \, .
\end{align}
Applying $\partial_\xi$ to $I$ we have
\begin{align}
\label{xidxiI}
& \partial_\xi I(s,\xi) = ic_1 \big( I_1(s,\xi) + I_2(s,\xi) \big) \, ,
\\
\label{I_1}
I_1(s,\xi) & = \iint e^{is \phi(\xi,\eta,\s)} {|\eta|}^{-2} 
  \partial_\xi \what{f}(s,\xi-\eta) \widehat{f}(s,\eta+\s) \overline{\widehat{f}}(s,\s)  d\eta d\s  \, ,
\\
\label{I_2}
I_2(s,\xi) & = is \iint e^{is \phi(\xi,\eta,\s)}  m(\xi,\eta) {|\eta|}^{-2}
  \what{f}(s,\xi-\eta) \widehat{f}(s,\eta+\s) \overline{\widehat{f}}(s,\s) \, d\eta d\s  \, ,
\end{align}
where we have denoted 
\begin{align}
 \label{m}
m(\xi,\eta) := \partial_\xi \Big( -\Lambda(\xi) + \Lambda(\xi-\eta) \Big)
  = -\Lambda^\p(\xi) \frac{\xi}{|\xi|} + \Lambda^\p(\xi-\eta) \frac{\xi-\eta}{|\xi-\eta|} \, . 
\end{align}
We will crucially use the fact that $m(\xi,\eta) \sim \eta$, for small $|\eta|$.
In particular, this will allow us to cancel part of the singularity given by the transform of the Coulomb potential $|\eta|^{-2}$,
so to compensate for the growing factor $s$ present in $I_2$.

\vskip10pt
\subsubsection{Estimate of \eqref{I_1}}
This term can be directly estimated using the Hardy-Littlewood-Sobolev inequality and the a priori assumptions \eqref{apriori}:
\begin{align*}
{\| {\langle \xi \rangle}^2 I_1(s) \|}_{L^2} & \les {\| \N \big(u(s), u(s), e^{is\Lambda} xf(s) \big) \|}_{H^2}
  \les {\| x f(s) \|}_{H^2} {\| u(s) \|}_{H^2} {\| u(s) \|}_{L^6} \les \e_1^3 {(1+s)}^{-1+p_0} \, .
\end{align*}

\vskip10pt
\subsubsection{Estimate of \eqref{I_2}}
%The term $I_2$  is more delicate. 
With the notation \eqref{phi^m_k}, we perform dyadic decomposition
in the variables $\xi,\eta$ and $\xi-\eta$, and write
\begin{align}
\label{I_2k}
\begin{split}
{\langle \xi \rangle}^2 I_2(s,\xi) & = \sum_{k,k_1,k_2 \in \Z} \langle \xi \rangle I_2^{k,k_1,k_2}(s,\xi)
\\
I_2^{k,k_1,k_2}(s,\xi) & := is \iint e^{is \phi(\xi,\eta,\s)} m_1(\xi,\eta)
  \what{f_{k_1}}(s,\xi-\eta)  \what{P_{k_2} |u|^2}(s,\eta) \, d\eta \, ,
\\
m_1(\xi,\eta) & := \langle \xi \rangle \Big( -\Lambda^\p(\xi) \frac{\xi}{|\xi|} + \Lambda^\p(\xi-\eta) \frac{\xi-\eta}{|\xi-\eta|}\Big)  {|\eta|}^{-2}
  \varphi_{k}(\xi) \varphi_{[k_2-2,k_2+2]}(\eta) \, . %\varphi_{[k_1-2,k_1+2]}(\xi-\eta)
\end{split}
\end{align}
In what follows we shall always work under the assumption that the integral above is not zero,
and, in particular the sums are taken over those indexes $(k,k_1,k_2)$ satisfying 
either $|k - \max\{k_1,k_2\}| \leq 10$ or $|k_1-k_2| \leq 10$.

%To obtain \eqref{concW1} it suffices to show
%\begin{align*}
%\sum_{k,k_1,k_2 \in \Z} 2^{k_+} {\| I_2^{k,k_1,k_2}(s)\|}_{L^2} \les \e_1^3 {(1+s)}^{-2+p_0} \, .
%\end{align*}

Using $|m_1(\xi,\eta)| \les 2^{-k_2}$, we see that
\begin{align}
 \label{I_2est0}
{\| I_2^{k,k_1,k_2}(s) \|}_{L^2} & \les 
  s \, 2^{3k/2} {\| f_{k_1}(s) \|}_{L^2} 2^{-k_2} {\| P_{k_2} |u(s)|^2 \|}_{L^2}  \, .
\end{align}
From the a priori assumptions \eqref{apriori} we know that
\begin{align}
\label{I_2est01}
& {\| f_{k_1}(s) \|}_{L^2} \les 2^{3k_1/2} 2^{-10 (k_1)_+} \e_1 \, .
\\
\label{I_2est02}
& {\| P_{k_2} |u(s)|^2 \|}_{L^2} \les \min \Big\{ 2^{3k_2/2}, 2^{-5k_2} \Big\} \e_1^2 \, .
\end{align}
Using \eqref{I_2est01} and \eqref{I_2est02} in \eqref{I_2est0}, we see that the sum over those indexes $k$
such that $2^k \leq {(1+s)}^{-2}$ can be easily dealt with:
\begin{align*}
\sum_{\substack{ k,k_,k_2 \in \Z \\ 2^k \leq {(1+s)}^{-2} } } {\| I_2^{k,k_1,k_2}(s) \|}_{L^2} 
  \les \e_1^3 {(1+s)}^{-2} \, .
\end{align*}
%In view of the above, we will assume below that summations are over indexes $k$ with $2^k \geq {(1+s)}^{-2}$.

From the definition in \eqref{I_2k}  %and $|\nabla^\alpha \Lambda(x) | \les_{|\a|} 1$, 
one can verify that
%\begin{align}
%\big| \partial_\xi^{a} \partial_\eta^{b} m_1(\xi,\eta) \big| \les 2^{-k_2} 2^{-|a| k} 2^{-|b| k_2} \, ,
%\end{align}
%$a,b \in \Z_+^3$ with $|a|,|b| \leq 10$. This implies that 
$m_1$ satisfies the hypothesis of Lemma \ref{lemprod} with $A \les 2^{-k_2}$.
Applying \eqref{touse3} we obtain
\begin{align}
 \label{I_2sum}
\Big\| \sum_{ \substack{ k,k_,k_2 \in \Z \\ 2^k \geq {(1+s)}^{-2} } } \langle \xi \rangle I_2^{k,k_1,k_2}(s) \Big\|_{L^2} 
  & \les s \sum_{\substack{ k,k_,k_2 \in \Z \\ 2^k \geq {(1+s)}^{-2} } } 2^{k_+} 
  {\| P_{k_1} f(s) \|}_{L^2} 2^{-k_2} {\| P_{k_2} |u(s)|^2 \|}_{L^\infty} \, .
\end{align}
%From the a priori estimates \eqref{apriori} we know that
%\begin{align}
%\label{I_2est1}
%{\| P_{k_1} f(s) \|}_{L^2} \les 2^{3k_1/2} 2^{-10 (k_1)_+} \e_1 \, .
%\end{align}
Using \eqref{aprioridecay} we see that
\begin{align}
\label{I_2est2}
{\| P_{k_2} |u(s)|^2 \|}_{L^\infty} \les \min \Big\{ 2^{3k_2}, {\big(1+2^{2k_2}\big)}^{-1} {(1+s)}^{-3} \Big\} \e_1^2 \, .
\end{align}
Therefore, in view of \eqref{I_2est01} and \eqref{I_2est2}, we can bound the sum in \eqref{I_2sum} for $2^{k_2} \leq {(1+s)}^{-1}$ by
\begin{align*}
\sum_{ \substack{ 2^k \geq {(1+s)}^{-2}  \\ 2^{k_2} \leq {(1+s)}^{-1}} } 2^{k_+} 
    2^{3k_1/2} 2^{-10 (k_1)_+} \,  2^{2k_2} \e_1^3 \les \e_1^3 {(1+s)}^{-2+p_0} \, .
\end{align*}
To estimate the sum in \eqref{I_2sum} for $2^{k_2} \geq {(1+s)}^{-1}$ we use again \eqref{I_2est01} and \eqref{I_2est2} to obtain
\begin{align*}
\sum_{ \substack{2^k \geq {(1+s)}^{-2} \\ 2^{k_2} \geq {(1+s)}^{-1}} } 2^{k_+} 
    2^{3k_1/2} 2^{-10 (k_1)_+}  \, 2^{-k_2} {\big(1+2^{2k_2}\big)}^{-1} {(1+s)}^{-3} \e_1^3 
    \les \e_1^3 {(1+s)}^{-2+p_0} \, .
\end{align*}
The last two inequalities imply the desired bound \eqref{concW1} for the term $I_2$, and give us \eqref{estW1}.

\vskip10pt
\subsection{Proof of \eqref{estW2}}\label{secweighted2}
In order to estimate ${\langle x \rangle}^2 f$ in $H^2$ we compute the contributions from $\partial_\xi^2 I$.
Applying $\partial_\xi$ to the terms $I_1$ and $I_2$ as they appear in \eqref{xidxiI}-\eqref{m}, 
and with a slight abuse of notation, we can write
\begin{align}
\label{xidxiI1}
& \partial_\xi \big( I_1(s,\xi) + I_2(s,\xi) \big)= J_1(s,\xi) + 2 J_2(s,\xi) + J_3(s,\xi) + J_4(s,\xi) \, ,
\\
\label{J_1}
J_1(s,\xi) & := \iint e^{is \phi(\xi,\eta,\s)} {|\eta|}^{-2} 
  \partial_\xi^2 \what{f}(s,\xi-\eta) \widehat{f}(s,\eta+\s) \overline{\widehat{f}}(s,\s)  d\eta d\s  \, ,
\\
\label{J_2}
J_2(s,\xi) & :=  is \iint e^{is \phi(\xi,\eta,\s)}  m(\xi,\eta)  {|\eta|}^{-2} 
  \partial_\xi \what{f}(s,\xi-\eta) \widehat{f}(s,\eta+\s) \overline{\widehat{f}}(s,\s) \, d\eta d\s  \, ,
\\
\label{J_3}
J_3(s,\xi) & := is \iint e^{is \phi(\xi,\eta,\s)} \partial_\xi m(\xi,\eta)  {|\eta|}^{-2} 
    \what{f}(s,\xi-\eta) \widehat{f}(s,\eta+\s) \overline{\widehat{f}}(s,\s) \, d\eta d\s  \, ,
\\
\label{J_4}
J_4(s,\xi) & := - s^2 \iint e^{is \phi(\xi,\eta,\s)} {[m(\xi,\eta)]}^2 {|\eta|}^{-2} 
  \what{f}(s,\xi-\eta) \widehat{f}(s,\eta+\s) \overline{\widehat{f}}(s,\s)  d\eta d\s  \, ,
\end{align}
where $m$ is defined in \eqref{m}.
To obtain \eqref{estW2} it is then enough to show that for $i=1,\dots,4$, one has
\begin{align}
\label{concW2}
& {\big\| {\langle \xi \rangle}^2 J_i(s)  \big\|}_{L^2} \les \e_1^3 {(1+s)}^{-1 + 2p_0} \, .
\end{align}

\vskip10pt
\subsubsection{Estimate of \eqref{J_1}}\label{secJ_1}
This is the easiest term and can be directly estimated as follows:
\begin{align*}
{\| {\langle \xi \rangle}^2 J_1(s) \|}_{L^2} & \les {\| \N \big(u(s), u(s), e^{is\Lambda} x^2 f(s) \big) \|}_{H^2}  
  \\
& \les {\| x^2 f(s) \|}_{H^2} {\| u(s) \|}_{H^2} {\| u(s) \|}_{L^6} \les \e_1^3 {(1+s)}^{-1 + 2p_0} \, .
\end{align*}

\vskip10pt
\subsubsection{Estimate of \eqref{J_2}}
Similarly to what has been done above for $I_2$ in \eqref{I_2k}, we can write
\begin{align}
{\langle \xi \rangle}^2 J_2(s,\xi) =  i s \sum_{k,k_1,k_2 \in \Z}  \langle \xi \rangle \iint e^{is \phi(\xi,\eta,\s)} m_2(\xi,\eta) 
  \what{P_{k_1} xf}(s,\xi-\eta) \what{P_{k_2} |u|^2}(s,\eta) \, d\eta \, ,
\end{align}
where 
\begin{align}
\label{m_2}
m_2(\xi,\eta) & := \langle \xi \rangle m_1(\xi,\eta)
  = \langle \xi \rangle \Big( -\Lambda^\p(\xi)\frac{\xi}{|\xi|} + \Lambda^\p(\xi-\eta)\frac{\xi-\eta}{|\xi-\eta|} \Big) {|\eta|}^{-2} 
  \varphi_{k}(\xi) \varphi_{[k_2-2,k_2+2]}(\eta) \, .
\end{align}
%For all $a,b \in \Z_+^3$ with $|a|,|b| \leq 10$ one can verify that
%\begin{align}
%\label{estm_2}
%\big| \partial_\xi^{a} \partial_\eta^{b} m_2(\xi,\eta) \big| \les 2^{-k_2} 2^{-|a| k} 2^{- |b| k_2} \, .
%\end{align}
It is not hard to verify that $m_2$ satisfies the hypothesis of Lemma \ref{lemprod} with $A \les 2^{-k_2}$.
We can then apply \eqref{touse3} and obtain
\begin{align}
 \label{J_2sum}
{\| {\langle \xi \rangle}^2 J_2(s) \|}_{L^2} & \les 
  s \sum_{2^k \geq {(1+s)}^{-2} } 2^{k_+} {\| P_{k_1} x f(s) \|}_{L^2} 2^{-k_2} {\| P_{k_2} |u(s)|^2 \|}_{L^\infty}  
\\
 \label{J_2sum1}
& + s \sum_{2^k \leq {(1+s)}^{-2}} 2^{3k/2} {\| P_{k_1} x f(s) \|}_{L^2} \, 2^{-k_2} {\| P_{k_2} |u(s)|^2 \|}_{L^2} \, .
\end{align}

Using Bernstein's inequality and interpolating weighted norms we can estimate
\begin{align*}
{\| P_{k_1} x f(s) \|}_{L^2} 
& \les 2^{k_1/2} {\| x f(s) \|}_{L^{3/2}} \les 2^{k_1/2} {\| x f(s) \|}^{1/2}_{L^2} {\| x^2 f(s) \|}^{1/2}_{L^2}
  \\
& \les 2^{k_1/2} {(1+s)}^{3p_0/2} \e_1
\end{align*}
This and the a priori bounds \eqref{apriori} give us
\begin{align}
\label{J_2est1}
& {\| P_{k_1} x f(s) \|}_{L^2} \les \min \Big\{ {\big(1+2^{2k_1}\big)}^{-1} {(1+s)}^{p_0} , 2^{k_1/2} {(1+s)}^{3p_0/2} \Big\} \e_1
\\
\label{J_2est2}
& {\| P_{k_2} |u(s)|^2 \|}_{L^\infty} \les \min \Big\{ 2^{3k_2}, {\big(1+2^{2k_2}\big)}^{-1} {(1+s)}^{-3} \Big\} \e_1^2 \, .
\end{align}

Using \eqref{J_2est1} and \eqref{J_2est2}, we can estimate the sum in \eqref{J_2sum} for $2^{k_2} \leq {(1+s)}^{-1}$ as follows:
\begin{align}
\sum_{ \substack{ 2^k \geq {(1+s)}^{-2} \\ 2^{k_2} \leq {(1+s)}^{-1}} }
   2^{k_+}  \min \Big\{ {\big(1+2^{2k_1}\big)}^{-1} {(1+s)}^{p_0} , 2^{k_1/2} {(1+s)}^{3p_0/2} \Big\} 
   2^{2k_2} \e_1^3 
   \les \e_1^3 {(1+s)}^{-2+2p_0} \, .
\end{align}
%We now estimate the sum in \eqref{J_2sum} for $2^{k_2} \leq {(1+s)}^{-1}$. 
Using again \eqref{J_2est1}  and \eqref{J_2est2} we see that
\begin{align*}
\sum_{ \substack{ 2^k \geq {(1+s)}^{-2} \\ 2^{k_2} \geq {(1+s)}^{-1}} }
   2^{k_+}  \min \Big\{ {\big(1+2^{2k_1}\big)}^{-1} {(1+s)}^{p_0} , 2^{k_1/2} {(1+s)}^{3p_0/2} \Big\} 2^{-k_2} {(1+s)}^{-3} \e_1^3 
   \les \e_1^3 {(1+s)}^{-2+2p_0} \, .
\end{align*}
Eventually, we estimate similarly the sum in \eqref{J_2sum1}:
\begin{align*}
\sum_{ 2^k \leq {(1+s)}^{-2} }
   2^{3k/2}  \min \Big\{ {\big(1+2^{2k_1}\big)}^{-1} {(1+s)}^{p_0} , 2^{k_1/2} {(1+s)}^{3p_0/2} \Big\} 
   \min \{ 2^{-k_2} , 2^{2k_2} \} \e_1^3 
   \les \e_1^3 {(1+s)}^{-2} \, .
\end{align*}
The last three inequalities give us the desired bound \eqref{concW2} for the term $J_2$.

\vskip10pt
\subsubsection{Estimate of \eqref{J_3}}
We write
\begin{align}
\label{J_3sum}
{\langle \xi \rangle}^2 J_3(s,\xi) =  i s \sum_{k,k_1,k_2 \in \Z} \langle \xi \rangle \iint e^{is \phi(\xi,\eta,\s)} m_3(\xi,\eta) 
  \what{f_{k_1}}(s,\xi-\eta) \what{P_{k_2} |u|^2}(s,\eta) \, d\eta \, ,
\end{align}
where 
\begin{align*}
m_3(\xi,\eta) %:= \langle \xi \rangle  \partial_\xi  m(\xi,\eta) {|\eta|}^{-2} 
  = \langle \xi \rangle \partial_\xi
  \Big( -\Lambda^\p(\xi)\frac{\xi}{|\xi|} + \Lambda^\p(\xi-\eta)\frac{\xi-\eta}{|\xi-\eta|} \Big) {|\eta|}^{-2} 
  \varphi_{k}(\xi) \varphi_{[k_2-2,k_2+2]}(\eta) \, .
\end{align*}
%Notice that
%\begin{align}
%m_3(\xi,\eta) := \partial_\xi m_2(\xi,\eta) - \frac{\xi}{\langle \xi \rangle} m_2(\xi,\eta) \, .
%\end{align}
%where $m_2$ is the symbol in \eqref{m_2}.
%It is then not hard to see that %for all $a,b \in \Z_+^3$ with $|a|,|b| \leq 10$ one has
%\begin{align*}
%\big| \partial_\xi^{a} \partial_\eta^{b} m_3(\xi,\eta) \big| \les 2^{-k_2} 2^{-|a| k} 2^{- |b| k_2} \, .
%\end{align*}
Proceeding analogously to the estimates in the above two paragraphs, 
we can easily bound by $\e_1^3{(1+s)}^{-2}$
the summation in \eqref{J_3sum} over $k \in \Z$ such that $2^k \leq {(1+s)}^{-2}$.
For the remaining contribution we apply once again Lemma \ref{lemprod} to the symbol $m_3$, with $A \les 2^{-k_2}$,
and use the a priori bounds to deduce
\begin{align*}
& {\| {\langle \xi \rangle}^2 J_3(s) \|}_{L^2} 
\\
& \les 
  s \sum_{2^k \geq {(1+s)}^{-2}} 2^{k_+} {\| P_{k_1} f(s) \|}_{L^2} 2^{-k_2} {\| P_{k_2} |u(s)|^2 \|}_{L^\infty} 
  +  \e_1^3{(1+s)}^{-2}
\\
& \les s \sum_{2^k \geq {(1+s)}^{-2}} 2^{k_+}  2^{3k_1/2} 2^{-10 (k_1)_+} 2^{-k_2} 
  \min \big\{ 2^{3k_2}, \big(1+2^{2k_2}\big)^{-1} {(1+s)}^{-3} \big\} \e_1^3 
  +  \e_1^3{(1+s)}^{-2}
\\
& \les \e_1^3 {(1+s)}^{-1+p_0} \, .
\end{align*}
The last inequality has been deduced once again by separately analyzing the two cases 
$2^{k_2} \leq {(1+s)}^{-1}$ and $2^{k_2} \geq {(1+s)}^{-1}$.

\vskip10pt
\subsubsection{Estimate of \eqref{J_4}}\label{secJ_4}
We can write
\begin{align*}
{\langle \xi \rangle}^2 J_4(s,\xi) =  - s^2 \sum_{k,k_1,k_2 \in \Z} \langle \xi \rangle \iint e^{is \phi(\xi,\eta,\s)} m_3(\xi,\eta) 
  \what{f_{k_1}}(s,\xi-\eta) \what{P_{k_2} |u|^2}(s,\eta) \, d\eta \, ,
\end{align*}
where 
\begin{align*}
m_4(\xi,\eta) := {\Big( -\Lambda^\p(\xi)\frac{\xi}{|\xi|} + \Lambda^\p(\xi-\eta)\frac{\xi-\eta}{|\xi-\eta|} \Big)}^2 
  {|\eta|}^{-2} \varphi_{k}(\xi) \varphi_{[k_2-2,k_2+2]}(\eta) \, .
\end{align*}
%Considerations similar to the one above show that
%\begin{align*}
%\big| \partial_\xi^{a} \partial_\eta^{b} m_4(\xi,\eta) \big| \les 2^{-|a| k} 2^{- |b| k_2} \, ,
%\end{align*}
%for all $a,b \in \Z_+^3$ with $|a|,|b| \leq 10$.
For $k$ with $2^k \leq {(1+s)}^{-2}$ a bound of $\e_1^3{(1+s)}^{-1}$ can be obtained as before.
To estimate the remaining contribution we notice that we can apply Lemma \ref{lemprod} to $m_4$ with $A \les 1$,
and obtain
\begin{align*}
& {\| {\langle \xi \rangle}^2 J_4(s) \|}_{L^2} 
\\
& \les 
  s^2 \sum_{2^k \geq {(1+s)}^{-2}} 2^{k_+} {\| P_{k_1} f(s) \|}_{L^2} {\| P_{k_2} |u(s)|^2 \|}_{L^\infty}
  + \e_1^3{(1+s)}^{-1}
\\
& \les s^2 \sum_{2^k \geq {(1+s)}^{-2}} 2^{k_+}   2^{3k_1/2} 2^{-10 (k_1)_+}
  \min \big\{ 2^{3k_2}, \big(1+2^{2k_2}\big)^{-1} {(1+s)}^{-3} \big\} \e_1^3 
  + \e_1^3 {(1+s)}^{-1}
\\
& \les \e_1^3 {(1+s)}^{-1+p_0} \, .
\end{align*}
This concludes the proof of \eqref{concW2}, which together with \eqref{concW1} gives us Proposition \ref{proW}.

\vskip15pt
\section{Proof of Proposition \ref{proZ} and \ref{proscatt}}%: bounds on the $L^\infty_\xi$-norm and scattering}
\label{secLinfty}

The aim of this section is to control uniformly in time the key norm 
\begin{align*}
{\| {(1+|\xi|)}^{10} \widehat{f}(t) \|}_\infty
\end{align*}
and show, as a byproduct of the proof,  
modified scattering as stated in \eqref{estproscatt} in Proposition \ref{proscatt}.

Given a solution $u$ of \eqref{eq1}, satisfying the a priori bounds \eqref{apriori}, we define for any $t\in[0,T]$ and $\xi \in \R^3$
\begin{align}
\label{correction}
B(t,\xi) :=  c_0 \int_0^t \int_{\R^3} {\left| \frac{\xi}{\langle \xi \rangle} - \frac{\s}{\langle \s\rangle} \right|}^{-1} 
  {|\what{u}(s,\s)|}^2 \, d\s \, \varphi_s (\xi) \frac{ds}{s+1} \, , \quad c_0 := {(2\pi)}^{-3} \, ,
\end{align}
where 
\begin{align}
\label{phi_s}
 \varphi_s (\xi) := \varphi (\xi s^{-1/300}) \, ,
\end{align}
for a smooth compactly supported function $\varphi$ as the one described before \eqref{phi_k}.
We also define the modified profile
\begin{align}
\label{modprof}
\qquad g(\xi,t) := e^{iB(t,\xi)} \what{f}(t,\xi) \, .
\end{align}
Notice that $B$ is a well-defined and real-valued function. In particular $\what{f}$ and $g$ have the same $L^\infty_\xi$-norm.

We are going to prove the following:
\begin{proposition}\label{prohatf}
Assume that $f\in C([0,T]:H^N)$ satisfies the a priori bounds \eqref{apriori}, that is
\begin{align}
\label{apriori10}
\begin{split}
\sup_{t\in[0,T]} \Big[ {(1+t)}^{-p_0} {\| u(t)\|}_{H^N} + 
  {(1+t)}^{-p_0} {\left\| x f(t) \right\|}_{H^1} + {(1+t)}^{-2 p_0} {\left\| x^2 f(t) \right\|}_{H^2}
\\
  + \,\, {\big\| {(1+|\xi|)}^{10} \what{f}(t) \big\|}_{L^\infty}  \Big] \leq \e_1 \, .
\end{split}
\end{align}
Then, for some $p_1>0$,
\begin{align}
\label{keybound0}
\sup_{t_1 \leq t_2 \in[0,T]} {(1+t_1)}^{p_1} 
  {\big\| {(1+|\xi|)}^{10} \big( g(t_2,\xi) - g(t_1,\xi) \big) \big\|}_{L^\infty_\xi} \les \e_1^3 \, .
\end{align}
In particular
\begin{align}
\label{estproZ1}
\sup_{t \in[0,T] } {\big\| {(1+|\xi|)}^{10} \what{f}(t,\xi) \big\|}_{L^\infty_\xi} \leq \e_0 + C \e_1^3 \, .
\end{align}
\end{proposition}

It is clear that the above statement implies Propositions \ref{proZ} and \ref{proscatt}:
\eqref{estproZ1} gives Proposition \ref{proZ}, while \eqref{keybound0} implies \eqref{estproscatt} 
once we define 
\begin{align*}
f_+(\xi) := \lim_{t \rightarrow \infty} g(t,\xi) \, ,
\end{align*}
where the limit is taken in ${(1+|\xi|)}^{-10} L^\infty_{\xi}$.

For \eqref{keybound0} 
it suffices to prove that if $\, t_1\leq t_2\in[2^m-2,2^{m+1}]\cap[0,T]$, for some $m\in\{1,2\ldots\}$, then
\begin{equation}
\label{keybound}
{\big\| {(1+|\xi|)}^{10} ( g(t_2,\xi) - g(t_1,\xi) ) \big\|}_{L^\infty_\xi} \les \e_1^3 2^{-p_1m} \, .
\end{equation}
To prove this we start by looking at the nonlinear term $I$ in \eqref{inteq}, and after a change of variables we write it as
\begin{align}
\label{inteq10}
\begin{split}
I(s,\xi) & := i c_1 \iint_{\R^3 \times \R^3} e^{is \phi(\xi,\eta,\sigma)} {|\eta|}^{-2}
     \overline{\widehat{f}}(s,\xi+\eta+\sigma) \widehat{f}(s,\xi+\eta) \widehat{f}(s,\xi+\sigma) \, d\eta d\sigma 
\\
& \phi(\xi,\eta,\sigma) := - \Lambda(\xi)  + \Lambda(\eta+\xi) + \Lambda(\sigma+\xi) - \Lambda(\xi+\eta+\sigma) \, .
\end{split}
\end{align}
Let $l_0$ be the smallest integer larger than $-29m/40$,
\begin{align}
\label{l_0}
l_0 := \Big[ - \frac{29}{40} m \Big] + 1 \, ,
\end{align}
using the notation \eqref{phi^m_k}, we write
\begin{align}
\label{I}
\begin{split}
& I(s,\xi) = I_0 (s,\xi) + i c_1 \sum_{l_1 \in \Z \, , \, l_1 > l_0} I_{l_1}(s,\xi) \, ,
\\
& I_0 (s,\xi) := i c_1  \iint_{\R^3 \times \R^3} 
  e^{is \phi(\xi,\eta,\sigma)} {|\eta|}^{-2} \varphi_{l_0}^{(l_0)}(\eta)
  \widehat{f}(s,\xi+\eta) \widehat{f}(s,\xi+\sigma) \overline{\widehat{f}}(s,\xi+\eta+\sigma) \, d\eta d\sigma  \, ,
\\
& I_{l_1}(s,\xi) := \iint_{\R^3 \times \R^3} 
  e^{is \phi(\xi,\eta,\s)} {|\eta|}^{-2} \varphi_{l_1}^{(l_0)}(\eta)
  \what{f}(s,\xi+\eta) \what{f}(s,\xi+\s) \overline{\what{f}}(s,\xi+\eta+\s) \, d\eta d\s  \, .
\end{split}
\end{align}
The term $I_0$ is the one responsible for the correction to the scattering,
%(See the discussion in the intro ....), 
whereas $I-I_0$ is a remainder term, under the a priori assumptions \eqref{apriori10}.

The profile $f$ verifies
\begin{align}
 \partial_t f(t,\xi) = I_0(s,\xi) + \sum_{l_1 > l_0} I_{l_1}(s,\xi)
\end{align}
and, according to the definition of $g$ above, we have
\begin{align}
 \partial_t g(t,\xi) = e^{iB(t,\xi)} \left[ I_0(t,\xi) + \sum_{l_1>l_0} I_{l_1}(s,\xi) - 
    i \partial_t B(t,\xi) \, \widehat{f}(t,\xi) \right] \, .
\end{align}
Therefore, to prove \eqref{keybound} it suffices to show that if $k \in \Z$, $m\in\{1,2,\ldots\}$, $|\xi|\in[2^k,2^{k+1}]$, 
and $t_1\leq t_2\in[2^m-1,2^{m+1}]\cap[0,T]$ the following two bounds are true:
\begin{align}
\label{keybound01}
& \left| \int_{t_1}^{t_2} e^{iB(s,\xi)}  \left[ I_0 (s,\xi) -
  i \partial_s  B(s,\xi) \, \what{f}(s,\xi) \right] \,ds \right| \les \e_1^3 2^{-p_1 m} 2^{-10k_+} \, ,
\\
\label{keybound02}
& \left| \int_{t_1}^{t_2} e^{iB(s,\xi)}  \sum_{l_1 > l_0} I_{l_1} (s,\xi)\,ds \right| \lesssim \e_1^3 2^{-p_1 m} 2^{-10k_+} \, .
\end{align}
From the definition of $B(s,\xi)$ in \eqref{correction}, we see that \eqref{keybound01}
can be reduced to the following two bounds
\begin{align}
\label{keybound1}
\begin{split}
& \Big|  I_0 (s,\xi) \big(1 - \varphi_s(\xi))  \Big| \les \e_1^3 2^{-(1+p_1)m} 2^{-10k_+} \, ,
\\
& \Big|  I_0 (s,\xi)\varphi_s(\xi) - i c_0 \int_{\R^3} {\left| \frac{\xi}{\langle \xi \rangle} - \frac{\s}{\langle \s\rangle} \right|}^{-1} 
  {|\what{f}(s,\s)|}^2 \, d\s \frac{\varphi_s (\xi) \what{f}(s,\xi)}{s+1}  \Big|
  \les \e_1^3 2^{-(1+p_1)m} 2^{-10k_+} \, .
\end{split}
\end{align}
For \eqref{keybound02} instead, it suffices to show
\begin{align}
\label{keybound2}
& \sum_{l_1 > l_0} \left|  I_{l_1,l_2} (s,\xi)\,ds \right| \les \e_1^3 2^{-(1+p_1) m} 2^{-10k_+} \, .
\end{align}
The estimates \eqref{keybound1} are proven in the next section, %more precisely in Lemma \ref{lemcorrhigh} and \ref{lemcorr},
while \eqref{keybound2} is proven in section \ref{secrem}.
%as a consequence of Lemma \ref{lemC0}, \ref{lemC1} and \ref{lemC2}.

%%%%%%%%%%%%%%%%%%%%%%%%%%
%%%% CORRECTION CASE  %%%%
%%%%      CASE D      %%%%
%%%%%%%%%%%%%%%%%%%%%%%%%%
\vskip10pt
\subsection{Proof of \eqref{keybound1}}\label{seccorr}

In view of the definition of $\varphi_s$ in \eqref{phi_s} , the first estimate in \eqref{keybound1} 
is a consequence of the following Lemma:

\begin{lem}[High frequency output]\label{lemcorrhigh}
Assume that \eqref{apriori} holds. Then, for all $m \in \{0,1,\dots\}$, $k\in \Z \cap  [m/300,\infty)$, 
$s \in [2^m-1,2^m]$ and $|\xi| \in [2^k,2^{k+1}]$
\begin{align}
& \big|  I_0 (s,\xi) \big| \lesssim \e_1^3 2^{-3m/2} 2^{-10k} \, .
\end{align}
\end{lem}

\begin{proof}
We decompose
\begin{align*}
\begin{split}
& I_0(s,\xi) = i c_1 \sum_{l_1 \leq l_0 + 10} I_0^{l_1}(s,\xi)
\\
& I_0^{l_1} (s,\xi) :=  \iint_{\R^3 \times \R^3} 
  e^{is \phi(\xi,\eta,\sigma)} {|\eta|}^{-2} \varphi_{l_1} (\eta) \varphi_{l_0}^{(l_0)}(\eta)
  \what{f}(s,\xi+\eta) \what{f}(s,\xi+\s) \overline{\what{f}}(s,\xi+\eta+\s) \, d\eta d\s \, . 
\end{split}
\end{align*}
Using the a priori bounds \eqref{apriori10} on the $L^\infty_\xi$-norm and on Sobolev norms, we can estimate
\begin{align}
 \label{est01}
| I_0^{l_1} (s,\xi) | \lesssim 2^{l_1} 2^{-10k} {\| f \|}_{H^{10}}^2 {\| {(1+|\xi|)}^{10} \what{f} \|}_{L^\infty} 
  \les  2^{l_1} 2^{2mp_0} 2^{-10k}  \e_1^3 \, .
\end{align}
On the other hand, using only the bounds on high Sobolev norms one can see that
\begin{align}
\label{est02}
| I_0^{l_1} (s,\xi) | \lesssim 2^{-l_1/2} 2^{-N k} {\| f \|}_{H^N}^3 \les  2^{-l_1/2} 2^{3 m p_0} 2^{-N k} \e_1^3 \, .
\end{align}
Using \eqref{est01} in the case $l_1 \leq -2m$, and \eqref{est01} for $l_1 \geq -2m$, together with $k \geq m/300$ and $N=1000$,
we obtain the desired conclusion.
\end{proof}

The next Lemma shows how to derive the correction term to the scattering, 
and proves the validity of the second inequality in \eqref{keybound1}.

\begin{lem}[Critical frequencies]\label{lemcorr}
Assume that \eqref{apriori} holds. Then, for all $m \in \{0,1,\dots\}$, $s \in [2^m-1,2^m]$,
and $k\in \Z \cap (-\infty, m/300]$, $|\xi| \in [2^k,2^{k+1}]$, we have
\begin{align}
\label{estlemcorr0}
 & \Big|  I_0 (s,\xi) - i c_0 \int_{\R^3} {\Big| \frac{\xi}{\langle \xi \rangle} - \frac{\s}{\langle \s\rangle} \Big|}^{-1} 
  {|\what{f}(s,\s)|}^2 \, d\s \, \frac{\what{f}(s,\xi)}{s+1}  \Big|
  \les \e_1^3 2^{-21m/20} \, .
\end{align}
\end{lem}

\begin{proof}
Let us recall the definition of $I_0$:
\begin{align}
\label{I_0}
\begin{split}
I_0 (s,\xi) & := i c_1  \iint e^{is \phi} {|\eta|}^{-2} \varphi_{l_0}^{(l_0)}(\eta)
  \widehat{f}(s,\xi+\eta) \widehat{f}(s,\xi+\sigma) \overline{\widehat{f}}(s,\xi+\eta+\sigma) \, d\eta d\sigma \, ,
\\
& \phi(\xi,\eta,\sigma) := - \Lambda(\xi)  + \Lambda(\eta+\xi) + \Lambda(\sigma+\xi) - \Lambda(\xi+\eta+\sigma) \, ,
\end{split}
\end{align}
where $\varphi_{l_0}^{(l_0)}$ is defined in \eqref{phi^m_k}, and $2^{l_0} \lesssim 2^{-29m/40}$.

\subsubsection*{Phase approximation}
In the support of the integral in \eqref{I_0} we can approximate the phase $\phi$ by a simpler expression.
Defining
\begin{align}
 \label{phi_0}
\phi_0(\xi,\eta,\s) :=  \eta \cdot \left( \frac{\xi}{\langle \xi \rangle} - \frac{\xi+\s}{\langle \xi + \s \rangle} \right) \, ,
\end{align}
we compute
\begin{align*}
\phi(\xi,\eta,\s) & = \frac{ |\eta|^2 + 2\eta \cdot \xi }{ \langle \xi \rangle + \langle \xi+\eta \rangle}
  + \frac{ - |\eta|^2 - 2\eta \cdot (\xi+\s) }{ \langle \xi+\s \rangle + \langle \xi+\eta +\s\rangle}
  = \frac{ \eta \cdot \xi }{ \langle \xi \rangle }
  + \frac{- \eta \cdot (\xi+\s) }{ \langle \xi+\s \rangle} + O \left( |\eta|^2 \right) \, ,
\end{align*}
and conclude
\begin{align}
\left| \phi(\xi,\eta,\s) - \phi_0(\xi,\eta,\s) \right| \lesssim |\eta|^2 \, .
\end{align}
Thus, if we let
\begin{align*}
 I_{0,1} (s,\xi) & := i c_1 \iint e^{is \phi_0(\xi,\eta,\s)} {|\eta|}^{-2} \varphi_{l_0}^{(l_0)}(\eta)
  \widehat{f}(s,\xi+\eta) \widehat{f}(s,\xi+\s) \overline{\widehat{f}}(s,\xi+\eta+\s) \, d\eta d\s
\end{align*}
it follows that for $s \in [2^m-1,2^m]$
\begin{align*}
& \left| I_0 (s,\xi) - I_{0,1} (s,\xi) \right| 
\\
& \lesssim \iint s |\phi(\xi,\eta,\s) - \phi_0(\xi,\eta,\s)| {|\eta|}^{-2} \varphi_{l_0}^{(l_0)}(\eta)
    | \widehat{f}(s,\xi+\eta)| \, |\widehat{f}(s,\xi+\s)| \, |\overline{\widehat{f}}(s,\xi+\eta+\s)| \, d\eta d\s 
\\ 
& \lesssim 2^m 2^{3l_0} %{(1+|\xi|)}^{-10} 
  {\| f(s) \|}_{L^2}^2 {\| \what{f}(s) \|}_{L^\infty} \les \e_1^3 2^{-11m/10} \, , %\e_1^3 2^{-(1+p_1)m} 2^{-10k_+} \, ,
\end{align*}
having used $ 3l_0 \les -21m/10$.

\subsubsection*{Profiles approximation}
We now want to further approximate $I_{0,1}$ by the expression
\begin{align}
\label{I_02}
I_{0,2} (s,\xi) & := i c_1 \iint e^{is \phi_0(\xi,\eta,\s)} {|\eta|}^{-2} \varphi_{l_0}^{(l_0)}(\eta)
  \widehat{f}(s,\xi) \widehat{f}(s,\xi+\s) \overline{\widehat{f}}(s,\xi+\s) \, d\eta d\s \, .
\end{align}
In order to do this let $\varphi$ and $\varphi_k$ be given be as in \eqref{phi_k}, and define
\begin{align*}
f_{\leq J}(x) & := \varphi_J^{(J)} (x) f(x) \qquad \mbox{for} \quad J \geq 1 
\\
f_{\geq J}(x) & := f(x) - f_{\leq J}(x) \, . 
\end{align*}
The function $f_{\leq J}$ is the restriction of $f$ to a ball centered at the origin and radius $\sim 2^J$ in real space.
 $f_{\geq J}$ is the portion of $f$ which lies at a distance greater than $2^J$ from the origin.
Using the a priori weighted bounds on $f$, we see that for $|\eta| \les 2^{l_0}$ one has
\begin{align*}
| \what{f}(\rho + \eta) - \what{f}(\rho) |  & \les 
  | \what{f_{\geq J}}(\rho + \eta) - \what{f_{\geq J}}(\rho) | + | \what{f_{\leq J}}(\rho + \eta) - \what{f_{\leq J}}(\rho) |
\\ 
  & \les {\| \what{f_{\geq J}} \|}_{L^\infty} +  {\| \partial \what{f_{\leq J}} \|}_{L^\infty} 2^{l_0}
\les 2^{-J/2} {\| x^2 f \|}_{L^2} + 2^{J/2} {\| x^2 f \|}_{L^2} 2^{l_0}
\\
& \les \e_1 2^{2mp_0} \big( 2^{-J/2} + 2^{J/2} 2^{l_0} \big) \, .
\end{align*}
Choosing $J = - l_0$ we obtain
\begin{align*}
| \what{f}(\rho + \eta) - \what{f}(\rho) |  & \les \e_1 2^{2mp_0} 2^{l_0/2} \, .
\end{align*}
From this, for $|\eta| \les 2^{l_0}$,  we see that
\begin{align*}
\int_{\R^3} \big| \widehat{f}(s,\xi+\eta) \widehat{f}(s,\xi+\s) \overline{\widehat{f}}(s,\xi+\eta+\s)
  - \widehat{f}(s,\xi) \widehat{f}(s,\xi+\s) \overline{\widehat{f}}(s,\xi+\s) \big| \, d\s \lesssim \e_1^3 2^{l_0/2} 2^{2mp_0} \, .
\end{align*}
As a consequence, for all $s \in [2^m-1,2^m]$, we obtain
\begin{align*}
& \left| I_{0,1} (s,\xi) - I_{0,2} (s,\xi) \right| \les 2^{3l_0/2} \e_1^3 2^{2mp_0} \les \e_1^3 2^{-21m/10} \, ,
\end{align*}
since $p_0 \leq 1/1000$.
%which is more than sufficient since $l_0 \les -4m/5$.

\subsubsection*{Final approximation}
To conclude the proof we need to show
\begin{equation}
\label{estlemcorr1}
\left| I_{0,2} (s,\xi) - i c_0 \int_{\R^3} {\left| \frac{\xi}{\langle \xi \rangle} - \frac{\s}{\langle \s\rangle} \right|}^{-1} 
  {|\what{u}(s,\s)|}^2 \, d\s \frac{\what{f}(s,\xi)}{s+1}  \right| \les \e_1^3 2^{-21m/20} \, ,%2^{-(1+p_1)m} 2^{-10k_+} \, .
\end{equation}
where $I_{0,2}$ is given by \eqref{I_02} and \eqref{phi_0}. %, and $B$ is defined in \eqref{correction}.
After a change of variables, \eqref{estlemcorr1} can be reduced to
\begin{equation}
\label{estlemcorr2}
\left| c_1 \iint e^{is \eta \cdot z} 
  {|\eta|}^{-2} \varphi_{l_0}^{(l_0)}(\eta) {| \what{f}(s,\s)|}^2 \, d\eta  d\s 
  - \frac{c_0}{s} \int {\left| z \right|}^{-1} {|\what{f}(s,\s)|}^2 \, d\s \right| 
\les \e_1^2 2^{-21m/20} \, 
\end{equation}
where we have defined
\begin{align}
\label{z}
 z := \frac{\xi}{\langle \xi \rangle} - \frac{\s}{\langle \s\rangle} \, .
\end{align}

%Since \eqref{estlemcorr2} is easily verified when the $d\s$ integral is performed over the region $|\s| \les 2^{-m/10}$,
%it then suffices to prove that for any fixed $\s \in \R^3$ with $|\s| \gtrsim 2^{-m/10}$, one has 

Observe that since $\F(|\eta|^{-2})(x) = 2\pi^2|x|^{-1}$, the following general formula holds for $x \in \R^3$:
\begin{align}
\label{formula}
\left| \int_{\R^3} e^{i \eta \cdot x} {|\eta|}^{-2} \varphi (\eta 2^{-l}) \, d\eta - 2\pi^2 {|x|}^{-1} \right| \les {|x|}^{-2} 2^{-l} \, .
\end{align}
Applying this with $x = sz$, $s\in[2^m-1,2^{m+1}]$, and $l=l_0 \gtrsim -29m/40$, gives us 
\begin{align*}
\left| c_1 \int e^{is \eta \cdot z} 
  {|\eta|}^{-2} \varphi_{l_0}^{(l_0)}(\eta) \, d\eta - \frac{c_0}{s} {\left| z \right|}^{-1} \right| 
\les  {|s z|}^{-2} 2^{-l_0} \les 2^{-5m/4} {|z|}^{-2} \, ,
\end{align*}
since $c_0 = {(2\pi)}^{-3}$ and $c_1 = 2{(2\pi)}^{-5}$.
Using this we can see that the left-hand side of \eqref{estlemcorr2} is bounded by
\begin{align}
 \label{estlemcorr3}
2^{-5m/4} \left| \int {|z|}^{-2} {|\what{f}(s,\s)|}^2 \, d\s \right| \, ,  
\end{align}
where $z$ is defined in \eqref{z}.
Since 
$|z| \gtrsim \min \{1, |\s|, |\xi-\s| {\langle \s \rangle}^{-3} \}$,
and $| {(1+|\xi|)}^{10} \what{f} |$ is a priori bounded in $L^\infty$, we see that
\begin{align*}
\left| \int {|z|}^{-2} {|\what{f}(s,\s)|}^2 \, d\s \right| \les \e_1^2 \, .
\end{align*}
Plugging this bound into \eqref{estlemcorr3} we obtain \eqref{estlemcorr2} and conclude the proof of the Lemma.
\end{proof}

\vskip10pt
\subsection{Proof of \eqref{keybound2}}\label{secrem}
We aim to prove:
\begin{align}
\label{keybound200}
\sum_{l_1 \in\Z, l_1 > l_0} \left|  I_{l_1} (s,\xi)\,ds \right| \les \e_1^3 2^{-(1+p_1) m} 2^{-10k_+} \, .
\end{align}
for $I_{l_1}$ defined in \eqref{I}, $s \in [2^m-1,2^{m+1}]$ with $m \in \{0,1, \dots\}$, $|\xi| \approx 2^k$ with $k \in \Z$,
and $l_0$ given by \eqref{l_0}. 
We decompose in dyadic pieces all the profiles and write
\begin{align}
\label{Isum}
\begin{split}
&  I_{l_1} = \sum_{k_1,k_2,k_3 \in \Z} I_{l_1}^{k_1,k_2,k_3}(s,\xi)
\\
& I_{l_1}^{k_1,k_2,k_3} (s,\xi) := \iint_{\R^3 \times \R^3} 
  e^{is \phi(\xi,\eta,\sigma)} {|\eta|}^{-2} \varphi_{l_1}^{(l_0)}(\eta)
  \what{f_{k_1}}(s,\xi+\eta) \overline{\what{f_{k_2}}}(s,\xi+\eta+\s) \what{f_{k_3}}(s,\xi+\s) \, d\eta d\s  \, .
\end{split}
\end{align}
We then aim to show
\begin{align}
\label{keybound300}
\sum_{l_1 > l_0, k_1,k_2,k_3 \in \Z} \left| I_{l_1}^{k_1,k_2,k_3}(s,\xi)\,ds \right| \les \e_1^3 2^{-(1+p_1) m} 2^{-10k_+} \, .
\end{align}

For all $s \in [2^m-1,2^{m+1}]$, we can estimate
\begin{align}
 \label{estI1}
| I_{l_1}^{k_1,k_2,k_3} (s,\xi) | & \les 2^{-l_1/2}  
  {\| \what{f_{k_1}} \|}_{L^2} {\| \what{f_{k_2}} \|}_{L^2} {\| \what{f_{k_3}} \|}_{L^2} \, .
\end{align}
Using the a priori bounds \eqref{apriori10} we know that
\begin{align}
 {\| \what{f_{k}} \|}_{L^2} \les \min \big\{ 2^{3k/2} 2^{-10k_+} , 2^{Nk_+} 2^{mp_0} \big\} \, .
\end{align}
Since $l_1 > l_0 \geq -3m/4$, and $N=1000$,
the last two estimates above suffice to show that the sum in \eqref{keybound300}
over those indexes $(k_1,k_2,k_3)$ with $\max\{k_1,k_2,k_3\} \geq m/300$ or $\min\{k_1,k_2,k_3\} \leq -m$,
satisfies the desired bound.
Since $l_1 \les \max\{k_1,k_2\}$, the remaining indexes in the sum are $O(m^4)$. 
The bound \eqref{keybound300} can then be reduced to showing that for
$s \in [2^m-1,2^{m+1}]$ with $m \in \{0,1, \dots\}$, and $|\xi| \approx 2^k$ with $k \in \Z \cap (-\infty, m/300 + 10]$, one has
\begin{align}
 \label{keybound400}
| I_{l_1}^{k_1,k_2,k_3} (s,\xi) | & \les \e_1^3 2^{-(1 + 2p_1) m} 2^{-10k_+} \, .
\end{align}
for fixed triples $(k_1,k_2,k_3)$ with
\begin{align*}
 -m \leq k_1,k_2,k_3 \leq m/300 \, .
\end{align*}

Let us further decompose
\begin{align}
\label{Isum2}
\begin{split}
&  I_{l_1}^{k_1,k_2,k_3}(s,\xi) = \sum_{l_2 \in \Z} I_{l_1,l_2}^{k_1,k_2,k_3}(s,\xi)
\\
& I_{l_1,l_2}^{k_1,k_2,k_3} (s,\xi) := \iint%_{\R^3 \times \R^3} 
  e^{is \phi(\xi,\eta,\sigma)} {|\eta|}^{-2} \varphi_{l_1}^{(l_0)}(\eta) \varphi_{l_2}(\s)
  \what{f_{k_1}}(s,\xi+\eta) \overline{\what{f_{k_2}}}(s,\xi+\eta+\s) \what{f_{k_3}}(s,\xi+\s) \, d\eta d\s  \, .
\end{split}
\end{align}
The above terms are zero if $l_2 \geq m/300 + 10$. Moreover, we can estimate
\begin{align*}
| I_{l_1,l_2}^{k_1,k_2,k_3} (s,\xi) | \lesssim  2^{-2l_1} 2^{-10 \max(k_1,k_2,k_3)_+} {\| \what{f}(s) \|}^3_{L^\infty}  
  \iint \varphi_{l_1}(\eta) \varphi_{l_2}(\s) \, d\eta d\s
  \les \e_1^3 2^{l_1} 2^{3l_2} 2^{-10k_+}  \, .
\end{align*}
This shows that
\begin{align*}
\sum_{l_2 \, : \, 3l_2 \leq -101m/100 - l_1} \left| I_{l_1,l_2}^{k_1,k_2,k_3}(s,\xi)\,ds \right| \les \e_1^3 2^{-(1+2p_1) m} 2^{-10k_+} \, .
\end{align*}
We are then left again with a summation over $l_2$ with only $O(m)$ terms.
Therefore, we see that \eqref{keybound200}, and hence  \eqref{keybound2}, will be a consequence of the following:

\vskip5pt
\begin{proposition}\label{pronr}
Let $I_{l_1,l_2}^{k_1,k_2,k_3}$ be defined as in \eqref{Isum2}, and assume \eqref{apriori10} holds.
Then, for all $s \in [2^m-1,2^{m+1}]$ with $m \in \{0,1, \dots\}$, $|\xi| \in [2^k,2^{k+1}]$ with $k \in \Z \cap (-\infty, m/300 + 10]$,
one has
\begin{align}
\label{keybound500}
\Big| I_{l_1,l_2}^{k_1,k_2,k_3} (s,\xi) \Big| \les \e_1^3 2^{-101m/100} \, ,
\end{align}
whenever
\begin{align}
 \label{freq0}
\begin{split}
-m \leq k_1,k_2,k_3 \leq m/300 \quad , \quad l_1 \geq - 29m/40 \quad \mbox{and}  \quad l_1+ 3l_2 \geq -101m/100 \, .
\end{split}
\end{align}
\end{proposition}
%The above Proposition will give us \eqref{keybound0} in Proposition \ref{prohatf}
The above Proposition is proven in a few steps, through Lemma \ref{lemC0}, \ref{lemC1} and \ref{lemC2} below.
We will always be under the assumption that \eqref{apriori10} holds, and all the indexes verify \eqref{freq0}.

\vskip5pt
\begin{lem}\label{lemC0}
The bound \eqref{keybound500} holds if
\begin{align}
\label{freqC0}
\max \{ k_1,k_2 \} \leq l_1 \, .
\end{align}
\end{lem}

\begin{proof}
We write
\begin{align*}
 I_{l_1,l_2}^{k_1,k_2,k_3} (s,\xi) = \iint e^{is \phi(\xi,\eta,\sigma)} m_1 (\eta,\s) 
  \what{f_{k_1}}(s,\xi+\eta) \overline{\what{f_{k_2}}}(s,\xi+\eta+\s) \what{f_{k_3}}(s,\xi+\s) \, d\eta d\s  \, ,
\end{align*}
where
\begin{align}
 \label{m_1}
m_1(\eta.\s) := {|\eta|}^{-2} \varphi_{l_1}^{(l_0)}(\eta) \varphi_{l_2}(\s) 
  \varphi_{[k_1-2,k_1+2]}(\xi+\eta) \varphi_{[k_2-2,k_2+2]}(\xi+\eta+\s) \, .
\end{align}
Since $m_1$ verifies the assumption of Lemma \ref{lemprod} with $A = 2^{-2l_1}$, we can estimate
\begin{align*}
\big| I_{l_1,l_2}^{k_1,k_2,k_3} (s,\xi) \big| \les 2^{-2l_1} {\| \what{f_{k_1}}(s) \|}_{L^2}
  {\| \what{f_{k_2}}(s) \|}_{L^2} {\| u_{k_3} (s) \|}_{L^\infty}
\\
\les 2^{-2l_1} 2^{k_1} \e_1  2^{k_2} \e_1 2^{-3m/2} \e_1 \les \e_1^3 2^{-3m/2} \, ,
\end{align*}
having used the a priori bounds on the $L^\infty$ norm of $\what{f}$ and the hypothesis $k_1,k_2 \leq l_1$ .
\end{proof}

\vskip5pt
\begin{lem}\label{lemC1}
Under the same assumptions of Proposition \ref{pronr}, the bound \eqref{keybound500} holds if
\begin{align}
\label{freqC1}
\max \{ k_1,k_2 \} \geq l_1 \qquad \mbox{and} \qquad |k_1 - k_2| \geq 10 \, .
\end{align}
\end{lem}

\begin{proof}
In this case we want to integrate by parts in $\eta$ in the expression \eqref{Isum2} for $I_{l_1,l_2}^{k_1,k_2,k_3}(s,\xi)$, 
using the identity
\begin{align}
 \label{m_0}
\begin{split}
& e^{is\phi(\xi,\eta,\s)} = \frac{1}{s} m_0(\xi,\eta,\s) \cdot \nabla_{\eta} e^{is\phi(\xi,\eta,\s)} \, , 
\\
& m_0(\eta,\s) := \frac{\nabla_\eta \phi(\xi,\eta,\s)}{i |\nabla_\eta \phi(\xi,\eta,\s)|^2} \, .
\end{split}
\end{align}
%Since $|k_1 - k_2| \geq 10$, one integration by parts will be sufficient to gain the necessary decay.
In particular, up to irrelevant constants, and with a slight abuse of notation, we can write
\begin{align}
\nn
& I_{l_1,l_2}^{k_1,k_2,k_3} (s,\xi) = I_1(s,\xi) + I_2(s,\xi)
\\
\label{C1I_1}
I_1(s,\xi) & = \iint e^{is \phi(\xi,\eta,\sigma)} \frac{1}{s} m_2 (\eta,\s) 
  \nabla_\eta \Big( \what{f_{k_1}}(s,\xi+\eta) \overline{\what{f_{k_2}}}(s,\xi+\eta+\s) \Big) \what{f_{k_3}}(s,\xi+\s) \, d\eta d\s  \, ,
\\
\label{C1I_2}
I_2(s,\xi) & = \iint e^{is \phi(\xi,\eta,\sigma)} \frac{1}{s} \nabla_\eta m_2 (\eta,\s) 
  \what{f_{k_1}}(s,\xi+\eta) \overline{\what{f_{k_2}}}(s,\xi+\eta+\s) \what{f_{k_3}}(s,\xi+\s) \, d\eta d\s  \, ,
\end{align}
where (omitting the variable $\xi$) we have denoted
\begin{align}
 \label{C1m_2}
m_2(\eta,\s) := m_0(\eta,\s) \frac{\varphi_{l_1}^{(l_0)}(\eta)}{{|\eta|}^{2}} \varphi_{l_2}(\s) 
  \varphi_{[k_1-2,k_1+2]}(\xi+\eta) \varphi_{[k_2-2,k_2+2]}(\xi+\eta+\s) \, .
\end{align}
One can then verify that $m_2$ satisfies the hypothesis of Lemma \ref{lemprod} with $A = 2^{-2l_1} 2^{-\max\{k_1,k_2\}}$,
so that we can apply \eqref{touse2} to estimate the term in \eqref{C1I_1} as follows:
\begin{align*}
\big| I_1(s,\xi) \big| & \les 2^{-m} 2^{-2l_1} 2^{-\max\{k_1,k_2\}} 
  \left[ {\| \partial \what{f_{k_1}}(s) \|}_{L^2} {\| \what{f_{k_2}}(s) \|}_{L^2} 
    + {\| \what{f_{k_1}}(s) \|}_{L^2} {\|  \partial \what{f_{k_2}}(s) \|}_{L^2} \right] {\| u_{k_3}(s) \|}_{L^\infty}
\\
& %\les 2^{-m} 2^{-2l_1} 2^{-\max\{k_1,k_2\}} 2^{mp_0} 2^{3\max\{k_1,k_2\}/2} 2^{-3m/2} \e_1^3
\les 2^{-m} 2^{-2l_1} 2^{mp_0} 2^{-3m/2} \e_1^3  
\les \e_1^3 2^{-101m/100} \, ,
\end{align*}
having used $-2l_1 \leq 29m/20$, and $p_0 \leq 1/1000$

For $I_2$ in \eqref{C1I_2} we perform an additional integration by parts and write (again up to irrelevant constants)
\begin{align}
\nn
& I_2(s,\xi) = J_1(s,\xi) + J_2(s,\xi)
\\
\label{C1J_1}
J_1(s,\xi) & = \iint e^{is \phi(\xi,\eta,\sigma)} \frac{1}{s^2} m_3 (\eta,\s) 
  \nabla_\eta \Big( \what{f_{k_1}}(s,\xi+\eta) \overline{\what{f_{k_2}}}(s,\xi+\eta+\s) \Big) \what{f_{k_3}}(s,\xi+\s) \, d\eta d\s  \, ,
\\
\label{C1J_2}
J_2(s,\xi) & = \iint e^{is \phi(\xi,\eta,\sigma)} \frac{1}{s^2} \nabla_\eta m_3 (\eta,\s) 
  \what{f_{k_1}}(s,\xi+\eta) \overline{\what{f_{k_2}}}(s,\xi+\eta+\s) \what{f_{k_3}}(s,\xi+\s) \, d\eta d\s  \, ,
\end{align}
where
\begin{align}
 \label{C1m_3}
m_3(\eta,\s) := m_0(\eta,\s) \nabla_\eta m_2(\eta,\s) \, .
\end{align}
From the definition of $m_0$ and $m_2$ in \eqref{m_0} and \eqref{C1m_2}, we see that
$m_3$ satisfies the hypothesis \eqref{touse1} in Lemma \ref{lemprod} with $A = 2^{-3l_1} 2^{-2\max\{k_1,k_2\}}$.
%for $|\xi+\eta| \approx 2^{k_1}$, $|\xi+\eta+\s| \approx 2^{k_2}$ with $|k_1 - k_2| \geq 10$
%\begin{align}
% \label{C1estm_3}
%| \partial^a_\eta \partial^b_\s m_3(\eta.\s) | \les 2^{-3l_1} 2^{-2\max\{k_1,k_2\}} 2^{-|a|l_1} 2^{-|b|l_1} \, .
%\end{align}
We then obtain
\begin{align*}
\big| J_1(s,\xi) \big| & \les 2^{-2m} 2^{-3l_1} 2^{-2\max\{k_1,k_2\}} 
 \left[ {\| \partial \what{f_{k_1}}(s) \|}_{L^2} {\| \what{f_{k_2}}(s) \|}_{L^2} 
    + {\| \what{f_{k_1}}(s) \|}_{L^2} {\|  \partial \what{f_{k_2}}(s) \|}_{L^2} \right] {\| u_{k_3}(s) \|}_{L^\infty}
\\
& \les 2^{-2m} 2^{-3l_1} 2^{-\max\{k_1,k_2\}/2} 2^{mp_0} 2^{-3m/2} \e_1^3  
\end{align*}
From the hypothesis \eqref{freq0} we see that $-3l_1 \leq 9m/4$, and $l_2 \geq - 2m/5$.
This latter implies 
$$-\max\{k_1,k_2\}/2 \leq -l_2/2 + 10 \leq m/5 + 10 \, ,$$ 
and therefore
\begin{align*} 
\big| J_1(s,\xi) \big| &
\les 2^{m/4} 2^{-\max\{k_1,k_2\}/2} 2^{mp_0} 2^{-3m/2} \e_1^3 \les 2^{-101m/100} \e_1^3 \, ,
\end{align*}
as desired.

To estimate $J_2$ in \eqref{C1J_2} we only use the pointwise bound
\begin{align}
| \nabla_\eta m_3(\eta,\s) | \les 2^{-4l_1} 2^{-2\max\{k_1,k_2\}}
\end{align}
and the a priori bounds \eqref{apriori10} to deduce
\begin{align*}
\big| J_2(s,\xi) \big| & \les 2^{-2m} 2^{-4l_1} 2^{-2\max\{k_1,k_2\}} 
  {\| \what{f_{k_1}}(s) \|}_{L^\infty} 2^{3l_1}  {\| \what{f_{k_2}}(s) \|}_{L^2}  {\| \what{f_{k_3}}(s) \|}_{L^2}
  2^{3l_1} 2^{3l_3}
\\
& \les 2^{-2m} 2^{-l_1} 2^{-\max\{k_1,k_2\}/2} \e_1^3
\les 2^{-101m/100} \e_1^3 \, ,
\end{align*}
having used once again $-l_1 \leq 3m/4$, and $-\max\{k_1,k_2\}/2 \leq m/5 + 10$.
\end{proof}

\vskip5pt
\begin{lem}\label{lemC2}
The bound \eqref{keybound500} holds if
\begin{align}
\label{freqC2}
|k_1 - k_2| \leq 10 \qquad \mbox{and} \qquad \max \{ k_1,k_2 \} \geq l_1 \, .
\end{align}
\end{lem}

\begin{proof}
The frequency configuration $k_1 \sim k_2$ is the most delicate.
Recall \eqref{m_0} and the notations
\begin{align}
 \label{C2m_0}
m_0(\eta,\s) & := \frac{\nabla_\eta \phi(\xi,\eta,\s)}{i |\nabla_\eta \phi(\xi,\eta,\s)|^2} 
\\
 \label{C2m_2}
m_2(\eta,\s) & := m_0(\eta,\s) \varphi_{l_1}^{(l_0)}(\eta) {|\eta|}^{-2} \varphi_{l_2}(\s) 
  \varphi_{[k_1-2,k_1+2]}(\xi+\eta) \varphi_{[k_2-2,k_2+2]}(\xi+\eta+\s) \, .
\\
 \label{C2m_3}
m_3(\eta,\s) & := m_0(\eta,\s) \nabla_\eta m_2(\eta,\s) \, .
\end{align}

Integrating by parts twice in the expression for $I_{l_1,l_2}^{k_1,k_2,k_3}$ in \eqref{Isum2}, 
or once more in \eqref{C1I_1}-\eqref{C1I_2}, we can write
\begin{align}
\nn
& I_{l_1,l_2}^{k_1,k_2,k_3} (s,\xi) = K_1(s,\xi) + K_2(s,\xi) +  K_3(s,\xi)
\\
\label{C2K_1}
K_1(s,\xi) & = \iint e^{is \phi(\xi,\eta,\s)} \frac{1}{s^2} q_1 (\xi,\eta,\s) 
  \nabla_\eta^2 \Big( \what{f_{k_1}}(s,\xi+\eta) \overline{\what{f_{k_2}}}(s,\xi+\eta+\s) \Big) \what{f_{k_3}}(s,\xi+\s) \, d\eta d\s  \, ,
\\
\label{C2K_2}
K_2(s,\xi) & = \iint e^{is \phi(\xi,\eta,\s)} \frac{1}{s^2} q_2 (\xi,\eta,\s) 
  \nabla_\eta \Big( \what{f_{k_1}}(s,\xi+\eta) \overline{\what{f_{k_2}}}(s,\xi+\eta+\s) \Big) \what{f_{k_3}}(s,\xi+\s) \, d\eta d\s  \, ,
\\
\label{C3K_2}
K_3(s,\xi) & = \iint e^{is \phi(\xi,\eta,\s)} \frac{1}{s^2} q_3 (\xi,\eta,\s) 
  \what{f_{k_1}}(s,\xi+\eta) \overline{\what{f_{k_2}}}(s,\xi+\eta+\s) \what{f_{k_3}}(s,\xi+\s) \, d\eta d\s  \, ,
\end{align}
where
\begin{align}
 \label{C2q_1}
q_1(\xi,\eta,\s) & := m_0(\xi,\eta,\s) m_2(\xi,\eta,\s) \, ,
\\
 \label{C2q_2}
q_2(\xi,\eta,\s) & :=  \nabla_\eta q_1(\xi,\eta,\s) + m_0(\xi,\eta,\s) \nabla_\eta m_2(\xi,\eta,\s) \, ,
\\
 \label{C2q_3}
q_3(\xi,\eta,\s) & :=  \nabla_\eta m_3(\xi,\eta,\s)  \, .
\end{align}

We now proceed to estimate the three integrals above.
First let us notice that for $|\xi+\eta| \approx 2^{k_1}$ and $|\xi+\eta+\s| \approx 2^{k_2}$ with $k_1 \sim k_2$, 
$|\eta| \approx 2^{l_1}$ and $|\s| \approx 2^{l_2}$, one has 
\begin{align}
 \label{C2estm_0}
| \partial^a_\eta \partial^b_\s m_0(\eta,\s)
% \varphi_{[l_1-5,l_1+5]}(\eta) \varphi_{[l_2-5,l_2+5]}(\s) \varphi_{[k_1-5,k_1+5]}(\xi+\eta) \varphi_{[k_2-5,k_2+5]}(\xi+\eta+\s) 
| 
\les 2^{-l_2} 2^{3\max\{k_1,k_2\}} 2^{-|a|l_1} 2^{-|b|l_2} \, ,
\end{align}
for $a,b \in \Z^3_+$ with $|a|,|b| \leq 10$.
As a consequence
\begin{align}
 \label{C2estm_2}
| \partial^a_\eta \partial^b_\s m_2(\eta,\s) | \les 2^{-2l_1} 2^{-l_2} 2^{3\max\{k_1,k_2\}} 2^{-|a|l_1} 2^{-|b|l_2} \, ,
\end{align}
for $a,b \in \Z^3_+$ with $|a|,|b| \leq 10$.
It then follows that
\begin{align}
 \label{C2estq_1}
{\| \F^{-1} q_1 \|}_{L^1} \les 2^{-2l_1} 2^{-2l_2} 2^{6\max\{k_1,k_2\}} \, .
%| \partial^a_\eta \partial^b_\s q_1(\eta,\s) | \les 2^{-2l_1} 2^{-2l_2} 2^{6\max\{k_1,k_2\}} 2^{-|a|l_1} 2^{-|b|l_1} \, 
\end{align}
We then apply Lemma \ref{lemprod} and obtain
\begin{align*}
\big| K_1(s,\xi) \big| & \les 2^{-2m} 2^{-2l_1} 2^{-2l_2} 2^{6\max\{k_1,k_2\}}
  {\| {\langle x \rangle}^2 f_{k_1} (s) \|}_{L^2} {\| {\langle x \rangle}^2 f_{k_2}(s) \|}_{L^2}  {\| u_{k_3}(s) \|}_{L^\infty}
\\
& \les 2^{-2m} 2^{-2l_1} 2^{-2l_2} 2^{4mp_0} 2^{2k_1} 2^{-3m/2}  \e_1^3
  \les 2^{-101m/100} \e_1^3 \, ,
\end{align*}
having used $-2l_1 \leq 29m/20$, $-2l_2 \leq 4m/5$, $k_1 \leq m/300$ and $p_0 \leq 1/1000$.

We can estimate similarly the term $K_2$ in \eqref{C3K_2}.
From the definition of $q_2$ in \eqref{C2q_1}-\eqref{C2q_2}, and the estimates \eqref{C2estm_0} and \eqref{C2estm_2} for $m_0$ and $m_2$,
we see that
\begin{align}
 \label{C2estq_2}
{\| \F^{-1} q_2 \|}_{L^1} \les 2^{-3l_1} 2^{-2l_2} 2^{6\max\{k_1,k_2\}} \, .
%| \partial^a_\eta \partial^b_\s q_2(\eta,\s) | \les 2^{-3l_1} 2^{-2l_2} 2^{6\max\{k_1,k_2\}} 2^{-|a|l_1} 2^{-|b|l_1} \, 
\end{align}
%for $a,b \in \Z^3_+$ with $|a|,|b| \leq 10$. 
Using \eqref{C2estq_2} and Lemma \ref{lemprod} we can obtain the bound
\begin{align*}
\big| K_2(s,\xi) \big| & \les 2^{-2m} 2^{-3l_1} 2^{-2l_2} 2^{6\max\{k_1,k_2\}}
  {\| {\langle x \rangle} f_{k_1} (s) \|}_{L^2} {\| {\langle x \rangle} f_{k_2}(s) \|}_{L^2}  {\| u_{k_3}(s) \|}_{L^\infty}
\\
& \les 2^{-2m} 2^{-3l_1} 2^{-2l_2} 2^{4mp_0} 2^{2k_1} 2^{-3m/2}  \e_1^3 \, .
%\\
%& \les 2^{-101m/100} \e_1^3 \, ,
\end{align*}
Now observe that the second constraint in \eqref{freq0} gives $- 2l_1 - 3m/2 \leq -m/20$.
Moreover, the second and third inequalities in \eqref{freq0} imply $-l_1-2l_2 \leq m$, 
as it can be seen, for instance, by considering the two cases $l_2 \geq -m/16$ and $l_2 \leq -m/16$.
From the chain of inequalities above we can then conclude that 
\begin{align*}
\big| K_2(s,\xi) \big| & \les 2^{-101m/100} \e_1^3 \, .
\end{align*}

Eventually we come to $K_3$. In this case we only use the pointwise bound for $q_3$
\begin{align*}
| \nabla_\eta q_3(\eta.\s) | \les 2^{-4l_1} 2^{-2l_2} 2^{6\max\{k_1,k_2\}} \, ,
\end{align*}
and estimate
\begin{align*}
\big| K_3(s,\xi) \big| & \les 2^{-2m} 2^{-4l_1} 2^{-2l_2} 2^{6\max\{k_1,k_2\}}
  {\| \what{f}_{k_1} (s) \|}_{L^\infty} {\| \what{f}_{k_2}(s) \|}_{L^\infty}  {\| \what{f}_{k_3}(s) \|}_{L^\infty} 2^{3l_1} 2^{3l_2}
\\
& \les 2^{-2m} 2^{-l_1} 2^{l_2} 2^{6\max\{k_1,k_2\}} \e_1^3
\\
& \les 2^{-101m/100} \e_1^3 \, ,
\end{align*}
having used once again the lower bound on $l_1$ in \eqref{freq0}, and $l_2,k_1,k_2 \leq m/300 + 10$.
\end{proof}

\appendix

\vskip15pt
\section{Subcritical Semi-relativistic Hartree equations}\label{secHg}

As already discussed in the introduction, some generalized models related to
the boson star equation \eqref{eq} have also been studied recently, and, in particular,
the class of semi-relativistic Hartree equations
\begin{align}
 \label{Hg}
i\partial_t u - \Lambda u = - \left( {|x|}^{-\gamma} \ast {|u|}^2 \right) u  
\quad , \quad \Lambda = \sqrt{1-\Delta} \quad , \quad x \in \R^n \quad , \quad 0 < \gamma < n \, .
\end{align}
We are interested here in constructing small scattering solutions when $\gamma > 1$.
For $\gamma > 2$, and $\gamma>3/2$ in the radial case, such solutions have been obtained in \cite{COSIAM06} and \cite{COSSDCDS09}.
Our proof of the weighted bounds in Proposition \ref{proW}, done for the case $\gamma=1$, can be adapted to prove the following:

\begin{theorem}\label{theoHg}
Let $u_0 : \R^3 \rightarrow \mathbb{C}$ be given such that
\begin{align*}
%\label{initdataHg}
{\| u_0 \|}_{H^{10}} + {\| {\langle x \rangle}^2 u_0 \|}_{H^3} \leq \e_0 \, .
\end{align*}
There exists $\bar{\e}_0$ such that for all $\e_0 \leq \bar{\e}_0$,
the Cauchy problem associated to \eqref{Hg} with $1< \gamma < 3$, 
with initial datum $u(t=0,x)=u_0(x)$, has a unique global solution satisfying
\begin{align}
\label{normHg}
\begin{split}
\sup_t \left[ {\| u(t) \|}_{H^{10}} + {\big\| {\langle x \rangle}^2 e^{it\Lambda} u(t) \big\|}_{H^3}  \right] \les \e_0 \, .
\end{split}
\end{align}
Furthermore, there exist $p_1 > 0$, and $f_+ \in L^2 ({\langle x \rangle}^4 dx)$, such that
\begin{align}
\label{theoHgscatt}
{\big\| e^{it\Lambda} u(t,x) - f_+ \big\|}_{L^2({\langle x \rangle}^4 dx)} \les \e_0 {(1+t)}^{-p_1} \, ,
\end{align}
for all $t>0$. A similar statement holds for $t<0$.
\end{theorem}

Since the above result follows from arguments similar to those in section \ref{secweighted}, 
we will just provide some ideas of its proof below.

\begin{proof}[Sketch of the proof of Theorem \ref{theoHg}]

Let us start by defining
\begin{align}
\label{Ng}
\N_\gamma (h_1, h_2, h_3) := \big( {|x|}^{-\gamma} \ast h_1 \overline{h_2} \big) h_3 \, ,
\end{align}
for $1 < \gamma < 3$.
The Hausdorff-Young inequality then gives
\begin{align}
\label{estNg}
{\| \N_\gamma (h_1, h_2, h_3) \|}_{L^2} \les {\| h_1 \|}_{L^{p_1}} {\| h_2 \|}_{L^{p_2}} {\| h_3 \|}_{L^{p_3}} \, ,
\end{align}
for any $p_1,p_2,p_3 \in [2,\infty]$ and
\begin{align}
\label{estNgp}
1/p_1+1/p_2+1/p_3 = 3/2 - \gamma/3 \, . 
\end{align}

We want to construct a global solution such that \eqref{normHg} holds.
Let us assume that we are given a local solution on $[0,T]$ which is a priori bounded as follows:
\begin{align}
\label{apriorig}
\sup_{t\in[0,T]} \left[ {\| u(t) \|}_{H^{10}} + {\big\| {\langle x \rangle}^2 e^{it\Lambda} u(t) \big\|}_{H^3}  \right] \les \e_1 \, ,
\end{align}
for some $\e_1 > 0$.
Duhamel's formula for $f(t) = e^{it\Lambda}u(t)$ reads:
\begin{align}
 \label{Duhamelg}
\begin{split}
u(t) = e^{-it\Lambda} u_0 + e^{-it\Lambda} N(t)  \, , \qquad
N(t) := \int_0^t e^{i\Lambda s} \N_\gamma (u(s),u(s),u(s)) \, ds \, .
\end{split}
\end{align}
To obtain a global solution it suffices to show that under the a priori assumptions \eqref{apriorig} one has
\begin{align}
\label{estI}
\sup_{t\in[0,T]} \left[ {\| N(t) \|}_{H^{10}} + {\big\| {\langle x \rangle}^2 N(t) \big\|}_{H^3}  \right] \les \e_1^3 \, ,
\end{align}
for some $C>0$.

Notice that \eqref{apriorig} implies, via the standard $L^p - L^q$ estimates
\begin{align}
\label{aprioridecayg}
\sup_{t\in[0,T]}  {(1+t)}^{3/2-3/p} {\| u(t) \|}_{L^p} \les \e_1 \, ,
\end{align}
for all $p \geq 2$.
Also notice than any global solution $u(t)$ which is bounded as in \eqref{normHg}, 
automatically scatters to a linear solution in $L^2$, because
\begin{align*}
{\| \N_\gamma (u(t), u(t), u(t)) \|}_{L^2} \les {\| u(t) \|}_{ L^{6/(3-\gamma)} }^2 {\| u(t) \|}_{L^2} \les \e_0^3 {(1+t)}^{-\gamma} \, ,
\end{align*}
which is an integrable function of time.
%Scattering in the stronger space $L^2 ({\langle x \rangle}^4 dx)$ will follow similarly, using the
%weighted estimates discussed below, and the integrable time decay of the nonlinearity given by $\gamma > 1$.

The first term in \eqref{estI} can be bounded directly by \eqref{estNg} and \eqref{aprioridecayg}:
\begin{align*}
{\| N(t) \|}_{H^{10}} \les  \int_0^t {\| u \|}_{ L^{6/(3-\gamma)} }^2 {\| u \|}_{H^{10}} \, ds
  \les \e_1^3 \int_0^t {(1+s)}^{-\gamma} \, ds \les \e_1^3  \, .
\end{align*}

To bound the second norm in the right-hand side of \eqref{estI} let us write $N(t)$ in Fourier space as
\begin{align}
 \label{whatIg}
\begin{split}
& \what{N}(t,\xi) = i c \int_0^t I(s,\xi) \, ds \, ,
\\
& I(s,\xi) = \iint_{\R^3 \times \R^3} e^{is \phi(\xi,\eta,\s)} {|\eta|}^{-3+\gamma}
     \what{f}(s,\xi-\eta) \what{f}(s,\eta+\s) \overline{\what{f}}(s,\s)  \, d\eta d\s  \, ,
\\
& \phi(\xi,\eta,\s) := - \Lambda(\xi)  + \Lambda(\xi-\eta) + \Lambda(\eta+\s) - \Lambda(\s) \, . 
\end{split}
\end{align}
Here $c = c(\gamma)$ denotes an appropriate positive constant which is irrelevant for the proof.
Then the idea is to proceed as in section \ref{secweighted2},
applying  $\langle \xi \rangle^3 \nabla_\xi^2$ to $I$ and estimating the resulting terms in $L^2$.
Applying $\nabla_\xi^2$ to $I$ we obtain four terms 
\begin{align}
\nn& \partial_\xi^2 I(s,\xi) = J_1^\gamma(s,\xi) + 2 J_2^\gamma(s,\xi) + J_3^\gamma(s,\xi) + J_4^\gamma(s,\xi) \, ,
\\
\label{J_1g}
J_1^\gamma(s,\xi) & := \iint e^{is \phi(\xi,\eta,\s)} {|\eta|}^{-3+\gamma} 
  \partial_\xi^2 \what{f}(s,\xi-\eta) \widehat{f}(s,\eta+\s) \overline{\widehat{f}}(s,\s)  d\eta d\s  \, ,
\\
\label{J_2g}
J_2^\gamma(s,\xi) & :=  is \iint e^{is \phi(\xi,\eta,\s)}  m(\xi,\eta)  {|\eta|}^{-3+\gamma} 
  \partial_\xi \what{f}(s,\xi-\eta) \widehat{f}(s,\eta+\s) \overline{\widehat{f}}(s,\s) \, d\eta d\s  \, ,
\\
\label{J_3g}
J_3^\gamma(s,\xi) & := is \iint e^{is \phi(\xi,\eta,\s)} \partial_\xi m(\xi,\eta)  {|\eta|}^{-3+\gamma} 
    \what{f}(s,\xi-\eta) \widehat{f}(s,\eta+\s) \overline{\widehat{f}}(s,\s) \, d\eta d\s  \, ,
\\
\label{J_4g}
J_4^\gamma(s,\xi) & := - s^2 \iint e^{is \phi(\xi,\eta,\s)} {[m(\xi,\eta)]}^2 {|\eta|}^{-3+\gamma} 
  \what{f}(s,\xi-\eta) \widehat{f}(s,\eta+\s) \overline{\widehat{f}}(s,\s)  d\eta d\s  \, ,
\end{align}
where the symbol $m$ is defined as in \eqref{m}:
\begin{align*}
m(\xi,\eta) := \partial_\xi \Big( -\Lambda(\xi) + \Lambda(\xi-\eta) \Big)
  %=  -\Lambda^\p(\xi) \frac{\xi}{|\xi|} + \Lambda^\p(\xi-\eta) \frac{\xi-\eta}{|\xi-\eta|} 
\, . 
\end{align*}
The terms $J_i^\gamma$ in \eqref{J_1g}-\eqref{J_4g}, for $i=1,\dots,4$, look exactly like the terms $J_i$, in \eqref{J_1}-\eqref{J_4},
with the exception that the power on $|\eta|$ is now $-3+\gamma > -2$.

To obtain \eqref{estI}, it would be sufficient to prove
\begin{align}
\label{concWg}
& {\big\| {\langle \xi \rangle}^3 J_i^\gamma (s) \big\|}_{L^2} \les \e_1^3 {(1+s)}^{-1-(\gamma-1)/4} \, .
\end{align}
To see this, one should proceed as in sections \ref{secJ_1}-\ref{secJ_4} and use the following two facts:

\noindent
1) Under the a priori assumptions \eqref{apriorig} one can set $p_0 = 0$ in all the estimates in \ref{secJ_1}-\ref{secJ_4}.

\noindent
2) Let $2^{k_2}$ denote the size of $|\eta|$ as in the estimates of sections \ref{secJ_1}-\ref{secJ_4}.
Since \eqref{J_1g}-\eqref{J_4g} have a factor ${|\eta|}^{-3+\gamma}$ instead of ${|\eta|}^{-2}$ as in \eqref{J_1}-\eqref{J_4}, 
one can obtain estimates for \eqref{J_1g}-\eqref{J_4g} 
which are a factor $2^{(\gamma-1)k_2}$ better than those for \eqref{J_1}-\eqref{J_4}.

\noindent 
Thanks to these observations, one can verify that the bounds \eqref{concW2} for the $L^2$ norms of $J_1,\dots,J_4$,
can be improved to the bounds \eqref{concWg} for $J_1^\gamma,\dots,J_4^\gamma$. This gives \eqref{estI}. 
The scattering statement \eqref{theoHgscatt} follows from the bounds \eqref{concWg} and the integrable time decay of $\N_\gamma$.
This concludes the proof the Theorem.
\end{proof}

\vskip15pt
\section{Auxiliary Estimates}

\vskip10pt
\subsection{Proof of Proposition \ref{prodecay}: Refined Linear Estimates}\label{seclinest}
In this section we give the proof of Proposition  \ref{prodecay} by showing
\begin{align}
\label{disperse}
{\left\| e^{i t \sqrt{1-\Delta}} f \right\|}_{L^\infty}
  \les \frac{1}{(1+|t|)^{3/2}} {\big\| {(1+|\xi|)}^6 \what{f}(\xi) \big\|}_{L^\infty_\xi}
  + \frac{1}{(1+|t|)^{31/20}} \Big[ {\big\| {\langle x \rangle}^2 f \big\|}_{L^2} + {\|f\|}_{H^{50}} \Big] \, .
\end{align}
for any $t \in \R$.
Estimate \eqref{disperse} is a simple but crucial ingredient in deriving the modified scattering behavior for solutions of \eqref{eq}.
It identifies the leading order norm that needs to be controlled in order to obtain the necessary sharp pointwise decay of $t^{-3/2}$,
and dictates what expression needs to be analyzed in order to capture the asymptotic behavior of solution of \eqref{eq}.
Similar estimates, as well as some variants, 
have been used when dealing with other $L^\infty$ critical equations (and not only),
see for example \cite{HN,DelortKG1d,HNmKdV,HNBO,KP,FNLS}.

Our proof is in the same spirit of the proof of Lemma 2.3 of \cite{FNLS}, where the author and Ionescu treated the
linear propagator $\exp(it{|\partial_x|}^{1/2})$. The analogous estimate for this propagator was then used 
to obtain global solutions to the gravity water waves problem in the case of one dimensional interfaces \cite{2dWW}.

\begin{proof}[Proof of \eqref{disperse}]
Set $\Lambda (\nabla):= \sqrt{1-\Delta} = \langle \nabla \rangle$.
Using the notation \eqref{phi^m_k}, we write
\begin{align}
\label{prop}
\begin{split}
e^{it \Lambda(\nabla)} f (x,t) = \sum_{k \in \Z} \int_{\R^3} e^{it \phi(\xi)} \what{f}(\xi) \varphi_k(\xi) \, d\xi 
\quad , \quad \phi(\xi;x,t) := \Lambda(\xi) + \xi \cdot \frac{x}{t} \, .
\end{split}
\end{align}
For \eqref{disperse} it then suffices to prove that
\begin{equation}
\label{disp1}
 \sum_{k \in \Z}
\Big| \int_{\R^3} e^{it \phi(\xi)} \what{f}(\xi) \varphi_k(\xi) \, d\xi \Big| \les 1 \, ,
\end{equation}
for any $t \in \R, x \in \R^3$, and any function $f$ satisfying
\begin{equation}
\label{disp2}
(1+|t|)^{-3/2} {\| {(1+|\xi|)}^{6} \what{f} \|}_{L^\infty_\xi} 
 + (1+|t|)^{-31/20} \big[ {\big\| {\langle x \rangle}^2 f \big\|}_{L^2} + {\|f\|}_{H^{50}} \big] \leq 1 \, .
\end{equation}

\subsubsection*{High and low frequencies}
Using only the bound ${\| \what{f} \|}_{L^\infty} \leq (1+|t|)^{3/2}$, 
we estimate first the contribution of small frequencies, 
\begin{equation*}
\sum_{2^k \les (1+|t|)^{-1/2}}
  \Big| \int_{\R^3} e^{it \phi(\xi)} \what{f}(\xi) \varphi_k(\xi) \, d\xi \Big|
  \les \sum_{2^k \les (1+|t|)^{-1/2}} 2^{3k} {\|\what{P_k f} \|}_{L^\infty} \les 1 \, .
\end{equation*}
Using instead the bound ${\| f \|}_{H^{50}} \leq (1+|t|)^{31/20}$, we can control the contribution of large frequencies:
\begin{equation*}
\sum_{2^k \gtrsim {(1+|t|)}^{1/30} } 
  \Big| \int_{\R^3} e^{it \phi(\xi)} \what{f}(\xi) \varphi_k(\xi) \,d\xi \Big|
  \les \sum_{2^k \gtrsim {(1+|t|)}^{1/30} } 2^{3k/2} {\| \what{P_kf} \|}_{L^2}
  \les \sum_{2^k \gtrsim {(1+|t|)}^{1/30} } 2^{-48k} {\| f \|}_{H^{50}} \les 1 \, .
\end{equation*}

\subsubsection*{Non-stationary frequencies}
From above we see that for \eqref{disp1} it suffices to prove
\begin{equation}
\label{disp4}
\sum_{ {(1+|t|)}^{-1/2} \les 2^k \les {(1+|t|)}^{1/30} } 
  \Big|\int_{\R^3} e^{it \phi(\xi)} \what{f}(\xi)\varphi_k(\xi)\,d\xi \Big| \les 1 \, .
\end{equation}
In proving \eqref{disp4} we may assume that $t \geq 1$. 
Notice that for $|x| < t$
\begin{align}
\nabla_\xi \phi (\xi) = 0 \quad \Longleftrightarrow \quad \xi = \xi_0 := \frac{x}{ \sqrt{t^2-|x|^2}} \, ,
\end{align}
while $|\nabla_\xi \phi| \gtrsim \Lambda(\xi)^{-2}$, for $|x| \geq t$.

We estimate first the non-stationary contributions when $\xi$ is away from $\xi_0$,
and more precisely when $2^{k} \geq 2^4|\xi_0|$ or $2^{k} \leq 2^{-4}|\xi_0|$.
In these cases we have $|\partial_r \phi | \gtrsim |\xi-\xi_0| {(1+2^{3k})}^{-1}$,
where $\partial_r = \xi/|\xi| \cdot \nabla_\xi \phi$ denotes the radial derivative.
We can integrate by parts twice in \eqref{prop} and write:
\begin{align}
\label{propIBP}
\begin{split}
& \int_{\R^3} e^{it \phi(\xi)} \what{f}(\xi) \varphi_k(\xi) \, d\xi  = I_k^{(1)} + I_k^{(2)} + I_k^{(3)} \, ,
\\
I_k^{(1)} & := -\frac{1}{t^2} \int_{\R^3} e^{it \phi(\xi)} {(\partial_r\phi})^{-2} \partial_r^2 \big( \what{f}(\xi) \varphi_k(\xi) \big) \, d\xi \, ,
\\
I_k^{(2)} & := -3\frac{1}{t^2} \int_{\R^3} e^{it \phi(\xi)} 
  {(\partial_r\phi})^{-1} \partial_r {(\partial_r\phi})^{-1} \partial_r \big( \what{f}(\xi) \varphi_k(\xi) \big) \, d\xi \, ,
\\
I_k^{(3)} & := \frac{1}{t^2} \int_{\R^3} e^{it \phi(\xi)} 
  \partial_r \big( {(\partial_r\phi})^{-1} \partial_r {(\partial_r\phi})^{-1} \big) \what{f}(\xi) \varphi_k(\xi) \, d\xi \, .
\end{split}
\end{align}

For $|\xi|\in [2^{k-2}, 2^{k+2}]$ with $2^{k} \geq 2^4|\xi_0|$ or $2^{k} \leq 2^{-4}|\xi_0|$,
one has $|\partial_r \phi | \gtrsim 2^k {(1+2^{3k})}^{-1}$. Therefore, using \eqref{disp2}  we can estimate
\begin{align*}
\big| I_k^{(1)} \big| & \les t^{-2} 2^{-2k} {(1+2^{3k})}^{2} {\| \partial_r^2 \big( \what{f} \varphi_k \big) \|}_{L^1}
  \les t^{-2} 2^{-2k} (1+2^{6k}) \big( 2^{3k/2} {\| x^2 f \|}_{L^2} + 2^{-2k} {\| \what{f_k} \|}_{L^1} \big)
\\
& \les t^{-2} \big( {(1+2^{3k})}^{2}  2^{-k/2} t^{31/20} + 2^{-k} t^{3/2} \big) \, ,
\end{align*}
and deduce that $\sum_k | I_k^{(1)} | \les 1$,
from the fact that we are only summing over those $k$ such that $t^{-1/2} \les 2^k \les t^{1/30}$,

To estimate $I_k^{(2)}$ in \eqref{propIBP},
we first notice that $|{(\partial_r\phi})^{-1} \partial_r {(\partial_r\phi})^{-1}| \les 2^{-3k} {(1+2^{3k})}^2$.
Moreover one has
\begin{align*}
{\| \partial_r \big( \what{f} \varphi_k \big) \|}_{L^1} \les 
  2^{2k} {\| \what{f_k} \|}_{L^\infty} + 2^{5k/2} {\| \partial_r \what{f} \|}_{L^6} 
  \les 2^{2k} {\| \what{f_k} \|}_{L^\infty} + 2^{5k/2} {\| x^2 f \|}_{L^2}  \, .
\end{align*}
Therefore, we see that
\begin{align*}
\big| I_k^{(2)} \big| & \les t^{-2} 2^{-3k} {(1+2^{3k})}^{2} {\| \partial_r \big( \what{f} \varphi_k \big) \|}_{L^1}
  \les t^{-2}  \big( 2^{-k} {\| \what{f} \|}_{L^\infty} + (1+2^{6k})2^{-k/2} {\| x^2 f \|}_{L^2} \big) \, .
\end{align*}
Using again \eqref{disp2} and the restrictions $t^{-1/2} \les 2^k \les t^{1/30}$, we get $\sum_k | I_k^{(2)} | \les 1$.

We can deal similarly with $I_3$.
Since $| \partial_r \big( {(\partial_r\phi})^{-1} \partial_r {(\partial_r\phi})^{-1} \big) | \les  2^{-4k} {(1+2^{3k})}^2$, we obtain
\begin{align*}
\sum_k \big| I_k^{(3)} \big| & \les t^{-2} \sum_{2^k \gtrsim t^{-1/2}} 2^{-4k} {(1+2^{3k})}^{2} {\| \what{f} \varphi_k \|}_{L^1}
  \les t^{-2} \sum_{2^k \gtrsim t^{-1/2}} 2^{-k} (1+2^{6k})  {\| \what{f}_k \|}_{L^\infty}  \les 1 \, .
\end{align*}

\subsubsection*{Stationary contributions}
To eventually conclude the proof of \eqref{disp4} it suffices to show
\begin{equation}
\label{disp5}
\Big| \int_{ \R^3} e^{it\phi(\xi)} \what{f}(\xi) \varphi_k(\xi)\,d\xi\Big| \lesssim 1,
\end{equation}
provided that $|t|\geq 1$, $(1+|t|)^{-1/2} \les 2^k \les (1+|t|)^{1/30}$, and $2^k \in [2^{-4}|\xi_0|, 2^4|\xi_0|]$.
%Clearly $|\xi_0|\approx 2^k$. 
Let $l_0$ denote the smallest integer with the property that $2^{l_0} \geq |t|^{-1/2}$ and estimate the 
left-hand side of \eqref{disp5} by
\begin{equation}
\label{disp6}
\Big| \int_{\R^3} e^{it \phi(\xi)} \what{f}(\xi) \varphi_k(\xi) \,d\xi \Big|
  \leq \sum_{l=l_0}^{k+100} |J_l| \, ,
\end{equation}
where, with the notation \eqref{phi_k}, for any $l\geq l_0$ we have defined
\begin{equation*}
J_{l} := \int_{\R^3} e^{it\phi(\xi)} \what{f_k}(\xi) \varphi_l^{(l_0)}(\xi-\xi_0) \, d\xi \, .
\end{equation*}

From \eqref{disp2} it immediately follows
\begin{equation*}
 |J_{l_0}| \les 2^{3l_0} {\|\what{f_k} \|}_{L^\infty} \les {t}^{-3/2} {\|\what{f} \|}_{L^\infty} \les 1 \, .
\end{equation*}

For $l > l_0$ we integrate by parts in the expression for $J_l$ above, relying on the fact that
$|\xi-\xi_0| \gtrsim 2^l \gtrsim t^{-1/2}$ on the support of the integral.
Two integration by parts like the ones performed in the previous paragraph, give
\begin{align}
\label{J_lIBP}
\begin{split}
& J_l = J_l^{(1)} + J_l^{(2)} + J_l^{(3)} \, ,
\\
J_l^{(1)} & := -\frac{1}{t^2} \int_{\R^3} e^{it \phi(\xi)} {(\partial_r\phi})^{-2} \partial_r^2 
  \big( \what{f_k}(\xi)\varphi_l^{(l_0)}(\xi-\xi_0) \big) \, d\xi \, ,
\\
J_l^{(2)} & := -3\frac{1}{t^2} \int_{\R^3} e^{it \phi(\xi)} 
  {(\partial_r\phi})^{-1} \partial_r {(\partial_r\phi})^{-1} \partial_r \big( \what{f_k}(\xi) \varphi_l^{(l_0)}(\xi-\xi_0) \big) \, d\xi \, ,
\\
J_l^{(3)} & := \frac{1}{t^2} \int_{\R^3} e^{it \phi(\xi)} 
  \partial_r \big( {(\partial_r\phi})^{-1} \partial_r {(\partial_r\phi})^{-1} \big) \what{f_k}(\xi) \varphi_l^{(l_0)}(\xi-\xi_0) \, d\xi \, .
\end{split}
\end{align}

Most of the above contributions  can be estimated in exactly the same way as we have estimated the terms in \eqref{propIBP},
using the fact that $|\partial_r^2 \phi(\xi)| \approx {(1+2^k)}^{-3}$ and $|\xi-\xi_0| \approx 2^l \geq t^{-1/2}$,
which imply $|\partial_r \phi(\xi)| \gtrsim {(1+2^k)}^{-3} 2^l$ in the support of $J_l$.
The term $J_l^{(3)}$, for example, verifies the exact same bound as $I_k^{(3)}$:
\begin{align*}
\big| J_l^{(3)} \big| & \les t^{-2}  2^{-4l} {(1+2^{3k})}^{2} {\| \what{f_k}(\cdot) \varphi_l(\cdot-\xi_0) \|}_{L^1}
  \les t^{-2} (1+2^{6k}){\| \what{f_k} \|}_{L^\infty} 2^{-l}  \les 1 \, ,
\end{align*}
Using again \eqref{disp2}, $2^l \geq t^{-1/2}$, and $2^k \les t^{1/30}$, we estimate
\begin{align*}
\big| J_l^{(1)} \big| & \les t^{-2} 2^{-2l} {(1+2^{3k})}^{2} {\| \partial_r^2 \big( \what{f}(\cdot) \varphi_l(\cdot - \xi_0) \big) \|}_{L^1}
\\
& \les t^{-2} 2^{-2l} (1+2^{6k}) \big( 2^{3l/2} {\| \partial_r^2 \what {f} \|}_{L^2} + 2^{-2l}  
  {\| \what{f_k} \mathbf{1}_{[0,2^{l+4}]}(|\xi-\xi_0|) \|}_{L^1} \big) %{\| \what{f}(\cdot) \varphi_l(\cdot - \xi_0) \|}_{L^1} \big)
\\
& \les t^{-2} 2^{-2l} (1+2^{6k}) \big( 2^{3l/2} {\| x^2 f \|}_{L^2} + 2^{l}  {\| \what{f_k} \|}_{L^\infty} \big)
 %{\| \what{f}(\cdot) \varphi_l(\cdot - \xi_0) \|}_{L^1} \big)
\\
& \les t^{-2}  \big( {(1+2^{3k})}^{2} 2^{-l/2} t^{31/20} + 2^{-l} t^{3/2} \big) \les 1 \, .
\end{align*}
Similarly
\begin{align*}
\big| J_l^{(2)} \big| & \les t^{-2} 2^{-3l} {(1+2^{3k})}^{2} {\| \partial_r \big( \what{f}(\cdot) \varphi_l(\cdot-\xi_0) \big) \|}_{L^1}
  \\
& \les t^{-2} (1+2^{6k}) \big( 2^{-l} {\| \what{f_k} \|}_{L^\infty} + 2^{-l/2} {\| x^2 f \|}_{L^2} \big) \les 1 \, .
\end{align*}
The desired bound \eqref{disp5} follows from \eqref{disp6},\eqref{J_lIBP} and the last three estimates. 
This completes the proof of \eqref{disp1} and the Proposition.
\end{proof}

\vskip10pt
\subsection{Bounds on pseudo-product operators}
Below we state a Lemma about pseudo-product operators which is used several times 
in the course of weighted energy estimates (section \ref{secweighted}) and remainder estimates (section \ref{secrem}).

\begin{lem}\label{lemprod}
Assume that $m\in L^1(\R^3 \times \R^3)$ satisfies
\begin{equation}\label{touse1}
{\left\| \int_{R^3 \times \R^3} m(\eta,\s) e^{ix\eta} e^{iy\s}\, d\eta d\s \right\|}_{L^1_{x,y}} \leq A \, ,
\end{equation}
for some $A \in (0,\infty)$. 
Then, for any $(p,q,r)$ with $1/p+1/q+1/r = 1$, %\in \{(2,2,\infty),(2,\infty,2),(\infty,2,2)\}$
\begin{equation}
\label{touse2}
 \left| \int_{\R^3 \times \R^3} \what{f}(\eta) \what{g}(\s) \what{h}(\eta+\s) m(\eta,\s) \,d\eta d\s \right|
  \les  A {\|f\|}_{L^p} {\|g\|}_{L^q} {\|h\|}_{L^r} \, .
\end{equation}

Moreover, for all $p,q$ with $1/p + 1/q = 1/2$, one has
\begin{equation}
\label{touse3}
 {\left\| \int_{\R^3} m(\xi,\eta)  \what{f}(\xi-\eta) \what{g}(\eta) \,d\eta \right\|}_{L^2_\xi}
  \les  A {\|f\|}_{L^p} {\|g\|}_{L^q}  \, .
\end{equation}

\end{lem}

\begin{proof}
We rewrite
\begin{equation*}
\begin{split}
\Big| \int_{\R \times \R} \what{f}(\eta) \what{g}(\s) \what{h} (-\eta-\s) m(\eta,\s) \,d\eta d\s \Big|
  &= C \Big| \int_{\R^3}f(x)g(y)h(z)K(z-x,z-y) \,dx dy dz \Big|
\\
  & \les \int_{\R^3} | f(z-x) g(z-y) h(z)| \,|K(x,y)| \, dx dy dz \, ,
\end{split}
\end{equation*}
where
\begin{equation*}
 K(x,y) := \int_{\R \times \R} m(\eta,\s) e^{ix\eta} e^{iy\s} \, d\eta d\s \, .
\end{equation*}
The desired bound \eqref{touse2} follows easily from \eqref{touse1} which says $K \in L^1_{x,y}$.
\eqref{touse3} follows from \eqref{touse2} by duality. \end{proof}

%In the energy estimate, we rely on the following Coifman-Meyer type theorem.
%\begin{lem}\label{penguin4}
%Consider a symbol $m$ homogeneous of order zero, i.e. $m(\xi,\eta) = m(\lambda \eta,\lambda \xi)$.
%It is therefore fully described by its trace on the one sphere $\{ |(\xi,\eta)|=1 \}$.
%Assume that $m$ enjoys the following properties:
%\begin{itemize}
%\item It is smooth on the sphere $\{ |(\xi,\eta)| = 1 \}$ away from the set 
%$\{\xi = 0\} \cup \{ \eta = 0 \} \cup \{ \xi - \eta = 0 \}$.
%\item Still on the sphere $\{ |(\xi,\eta)| = 1 \}$, but in the region ${|\xi|\ll |\eta|}$ it can be written
%\begin{align*}
% m(\xi,\eta) = \frac{\xi}{|\xi|} |\xi|^\alpha \widetilde{m}(\xi,\eta) \qquad 
%\mbox{or} \qquad m(\xi,\eta) = |\xi|^\alpha \widetilde{m}(\xi,\eta) 
%\end{align*}
%where $\alpha \geq 0$, and $m$ is smooth.
%\item The same property holds if one permutes the roles of $\xi$, $\eta$ and $\xi-\eta$.
%\end{itemize}
%Then if $\kappa>0$
%\begin{align*}
%{\| T_m(f,g) \|}_2 \lesssim {\|f\|}_{L^2} {\|(|D|^\kappa + |D|^{-\kappa}) g\|}_{L^\infty} \, .
%\end{align*}
%\end{lem}

\vskip20pt
%\addcontentsline{toc}{section}{Bibliography}
\bibliographystyle{plain}

\end{document}